\documentclass[a4paper,12pt]{amsart}

\usepackage{amssymb, amsmath, amsthm, mathrsfs, braket, xspace}
\usepackage[margin=1.0in]{geometry}

\numberwithin{equation}{section}
\usepackage[colorlinks=true]{hyperref}

\usepackage{lineno}
\pagewiselinenumbers
\nolinenumbers\renewcommand{\linelabel}[1]{}

\usepackage{color}

\newcommand{\red}{}
 
\usepackage{mathtools}
\mathtoolsset{showonlyrefs=true} 
\usepackage{enumitem}
\setlist[enumerate]{label=$(\mathrm{\arabic*})$}

\usepackage[all,2cell]{xy}
\objectmargin+{1mm}
\labelmargin+{0.8mm}
\SelectTips{cm}{12}
\usepackage{aliascnt}
\usepackage{comment}

\newtheorem{thm}{Theorem}[section]

\newaliascnt{cor}{thm}
\newtheorem{cor}[cor]{Corollary}
\aliascntresetthe{cor}

\newaliascnt{lem}{thm}
\newtheorem{lem}[lem]{Lemma}
\aliascntresetthe{lem}

\newaliascnt{prop}{thm}
\newtheorem{prop}[prop]{Proposition}
\aliascntresetthe{prop}

\newaliascnt{conj}{thm}
\newtheorem{conj}[conj]{Conjecture}
\aliascntresetthe{conj}

\theoremstyle{definition}
\newaliascnt{dfn}{thm}
\newtheorem{dfn}[dfn]{Definition}
\aliascntresetthe{dfn}

\newaliascnt{rem}{thm}
\newtheorem{rem}[rem]{Remark}
\aliascntresetthe{rem}

\newaliascnt{prob}{thm}

\aliascntresetthe{prob}

\newaliascnt{ex}{thm}
\newtheorem{ex}[ex]{Example}
\aliascntresetthe{ex}

\newcommand{\Q}{\mathbb{Q}}

\newcommand{\Z}{\mathbb{Z}}
\newcommand{\Zhat}{\widehat{\Z}}

\newcommand{\ba}{\mathbf{a}}
\newcommand{\bb}{\mathbf{b}}

\newcommand{\bA}{\mathbb{A}}
\newcommand{\bP}{\mathbb{P}}

\newcommand{\bZ}{\mathbb{Z}}

\newcommand{\bW}{\mathbb{W}}

\newcommand{\sC}{\mathscr{C}}

\newcommand{\sJ}{\mathscr{J}}
\newcommand{\sX}{\mathscr{X}}

\newcommand{\cO}{\mathcal{O}}
\newcommand{\cM}{\mathcal{M}}
\newcommand{\cN}{\mathcal{N}}
\newcommand{\cF}{\mathcal{F}}

\DeclareMathOperator{\Br}{Br}
\DeclareMathOperator{\ch}{ch}
\DeclareMathOperator{\Coker}{Coker}
\DeclareMathOperator{\Cor}{Cor}
\DeclareMathOperator{\dlog}{dlog}

\DeclareMathOperator{\Gal}{Gal}

\DeclareMathOperator{\Hom}{Hom}

\DeclareMathOperator{\Jac}{Jac}
\DeclareMathOperator{\Ker}{Ker}
\DeclareMathOperator{\res}{res}
\DeclareMathOperator{\Spec}{Spec}
\DeclareMathOperator{\Tr}{Tr}
\DeclareMathOperator{\tr}{tr}
\DeclareMathOperator{\Ztr}{\Z_{\mathrm{tr}}}

\newcommand{\ab}{\mathrm{ab}}
\newcommand{\AS}[1]{[\,#1\,)}
\newcommand{\et}{\mathrm{et}}

\newcommand{\geo}{\mathrm{geo}}
\newcommand{\Gm}{\mathbb{G}_{m}}
\newcommand{\Ga}{\mathbb{G}_{a}}
\newcommand{\Id}{\mathrm{id}}
\newcommand{\sep}{\mathrm{sep}}
\newcommand{\otimesM}{\overset{M}{\otimes}}

\newcommand{\Cork}{\mathbf{Cor}_k}
\newcommand{\Corkaff}{\mathbf{Cor}_k^\mathrm{aff}}

\newcommand{\uMCor}{\operatorname{\mathbf{{\underline{M}}Cor}}}

\newcommand{\MCor}{\operatorname{\mathbf{MCor}}}
\newcommand{\MCork}{\MCor_k}

\newcommand{\bcube}{{\overline{\square}}}
\newcommand{\cube}{\square}

\newcommand{\PSh}{\operatorname{PSh}}
\newcommand{\Sh}{\operatorname{Sh}}

\newcommand{\RSC}{{\operatorname{\mathbf{RSC}}}}
\newcommand{\HI}{{\operatorname{\mathbf{HI}}}}

\newcommand{\Nis}{\mathrm{Nis}}

\title{Extended differential symbol and the Kato homology groups}

\author[T. Hiranouchi]{Toshiro Hiranouchi}\address[T. Hiranouchi]{
Department of Basic Sciences, Graduate School of Engineering, 
Kyushu Institute of Technology, 
1-1 Sensui-cho, Tobata-ku, Kitakyushu-shi, 
Fukuoka 804-8550 JAPAN}
\email{hira@mns.kyutech.ac.jp}

\author[R. Sugiyama]{Rin Sugiyama} \address[R. Sugiyama]{
Department of Mathematics, Physics and Computer Science, 
Japan Women’s University, 2-8-1 Mejirodai, Bunkyo-ku Tokyo, 112-8681 JAPAN
}\email{sugiyamar@fc.jwu.ac.jp}


\begin{document}


\begin{abstract}
Building on our previous work \cite{Hir24}, 
we investigate an analogue of the differential symbol map used in the Bloch-Gabber-Kato theorem. 
Within this framework, for an appropriate variety over a field, the higher Chow group corresponds to the 0-th Kato homology group. 
Inspired by Akhtar's theorem \cite{Akh04c} on higher Chow groups, 
we investigate the structure of the 0-th Kato homology group for varieties over arithmetic fields, including finite fields, local fields, and global fields of positive characteristic.
We also express our results in terms of reciprocity sheaves.
\end{abstract}

\maketitle

\section{Introduction}
\label{sec:intro}
Let $F$ be a field of characteristic $p>0$. 
{\red \linelabel{com:1}Let $K_n^M(F)$ denote the $n$-th Milnor $K$-group of the field $F$, 
and 
\[
H^i(F,\Z/p^r\Z(j)) = H^{i-j}_{\et}(\Spec(F), W_r\Omega_{\log}^j) 
\]
denote the \'etale cohomology group $H^{i-j}_{\et}(\Spec(F), W_r\Omega_{\log}^j)$ 
of the logarithmic Hodge-Witt sheaf $W_r\Omega_{\log}^j$ on $\Spec(F)$ (for the precise definition, see \cite[Sect.~0]{Kat86}).}
The Bloch-Gabber-Kato theorem (\cite[Cor.~2.8]{BK86})  
and Kahn's theorem (\cite[Thm.~4.5]{Hir24})
state that the differential symbol map 
\begin{equation}
    \label{eq:BKG}
  s_{p^r}^n\colon (\Gm^{\otimesM n}(F))/p^r(\Gm^{\otimesM n}(F)) \xrightarrow{\simeq} K_n^M(F)/p^rK_n^M(F) \xrightarrow{\simeq} H^n(F,\Z/p^r\Z(n))
\end{equation}
is an isomorphism, 
for any $r\ge 1$ and $n\ge 0$ (see also \autoref{thm:BKG}, \autoref{thm:H24}). 
Here, 
$\Gm^{\otimesM n}(F)$ denotes the $n$-fold Mackey product of the multiplicative group $\Gm$ 
(see \autoref{def:otimesM} for the definition). 
In \cite{Hir24}, 
we extended the differential symbol map, resulting in the isomorphism
\begin{equation}
\label{eq:H24}    
\tilde{s}_{p^r}^n\colon \Big(W_r \otimesM \Gm^{\otimesM n}\Big) (F)/\wp \xrightarrow{\simeq} H^{n+1}(F,\Z/p^r\Z(n)) =: H^{n+1}_{p^r}(F),
\end{equation}
for $r\ge 1$ and $n\ge 0$. 
Here, an endomorphism $\wp\colon \Big(W_r \otimesM \Gm^{\otimesM n}\Big) (F) \to \Big(W_r \otimesM \Gm^{\otimesM n}\Big) (F)$ is defined by the Artin-Schreier-Witt map on the Witt group $W_r$,  
and 
$\Big(W_r \otimesM \Gm^{\otimesM n}\Big) (F)/\wp  := \Coker(\wp)$ 
(see \autoref{thm:H24}). 
{\red In particular,\linelabel{com:22'} when $n=1$, we have 
\begin{equation}\label{eq:HBr}
    H^2_{p^r}(F) \simeq \Br(F)[p^r],
\end{equation}
where $\Br(F)[p^r]$ denotes the $p^r$-torsion part $\Br(F)[p^r]$ of the Brauer group $\Br(F)$ of $F$
(cf.~\cite[Sect.~0]{Kat86}). 
Thus, the isomorphism \eqref{eq:H24} provides a representation of the Brauer group.}
In this paper, we utilize this extended differential symbol map to study the 0-th Kato homology groups.

To state our main theorem precisely, we 
assume $[F:F^p]\le p^s$ for some integer $s\ge 0$.
For an appropriate scheme $X$ over $F$ 
and an integer $r \ge 1$, 
there is a homological complex $KC_{\bullet}^{(s)}(X,\Z/p^r)$ of Bloch-Ogus type 
(\cite[Prop.~1.7]{Kat86}): 
\begin{equation}\label{eq:KC}
\vcenter{
\xymatrix@C=5mm@R=-2mm{
 \cdots \ar[r]^-{\partial} & \displaystyle\bigoplus_{x\in X_{j}}H^{s+ j+1}_{p^r}(F(x)) \ar[r]^-{\partial} &\cdots  
 \ar[r]^-{\partial}& \displaystyle\bigoplus_{x\in X_{1}}H^{s+2}_{p^r}(F(x)) \ar[r]^-{\partial}  & \displaystyle\bigoplus_{x\in X_{0}}H^{s+1}_{p^r}(F(x)), \\
 & \mbox{ \small degree $j$ } & & \mbox{\small degree $1$} & \mbox{\small degree $0$} 
}}
\end{equation}
where 
$X_j$ denotes the set of points $x$ in $X$ with $\dim \overline{\set{x}} = j$, and 
$F(x)$ denotes the residue field at $x$. 
For $j\ge 0$, 
the $j$-th \textbf{Kato homology group} of $X$ (with coefficients in $\Z/p^r\Z$) is defined 
by 
\begin{equation}
KH_j^{(s)}(X,\Z/p^r\Z) := H_j(KC_{\bullet}^{(s)}(X,\Z/p^r\Z))
\end{equation}
(see \autoref{sec:KH} for the precise definition). 
In particular, we have 
\begin{equation}
\label{eq:CokKH}
\Coker\left(\partial \colon \bigoplus_{x\in X_{1}}H^{s+2}_{p^r}(F(x)) \to \bigoplus_{x\in X_{0}}H^{s+1}_{p^r}(F(x)) \right) = KH_0^{(s)}(X,\Z/p^r\Z).
\end{equation}
The structure map $f\colon X\to \Spec(F)$ induces a map 
\[
f_{\ast}\colon KH_0^{(s)}(X,\Z/p^r\Z) \to KH_0^{(s)}(\Spec(F),\Z/p^r\Z) = H_{p^r}^{s+1}(F).
\]
One of our main results provides conditions 
under which the induced homomorphism $f_{\ast}$ is an isomorphism. 
This result includes the following theorem when $F$ is an arithmetic field, namely, 
a finite field, a local field, or a global field. 
Here, a \textbf{local field} is a completely discrete valuation field with finite residue field, 
and a \textbf{global field} (of positive characteristic) is a function field of one variable 
over a finite field.

\begin{thm}[{\autoref{thm:finite}, \autoref{thm:local}, and \autoref{thm:global}}]
\label{thm:main_intro}
Let $F$ be a field of characteristic $p>0$, 
and $X$ a projective smooth and geometrically irreducible scheme over $F$.  
 
    \begin{enumerate}
        \item If $F$ is a finite field, 
        then we have $KH_0^{(s)}(X,\Z/p^r\Z) = 0$ for $s\ge 1$, and 
        \[ 
        f_{\ast}\colon KH_0^{(0)}(X,\Z/p^r\Z) \to KH_0^{(0)}(\Spec(F),\Z/p^r\Z)
        \]
        is an isomorphism.
        \item If $F$ is a local field or a global field, 
        then we have $KH_0^{(s)}(X,\Z/p^r\Z) = 0$ for $s\ge 2$. 
        If we further assume that 
        $X= X_1\times \cdots \times X_d$ is the product of  curves $X_1,\ldots,X_d$ over $F$ with $X_i(F)\neq \emptyset$, 
        then the induced map 
        \[
        f_{\ast}\colon KH_0^{(1)}(X,\Z/p^r\Z) \xrightarrow{\simeq} KH_0^{(1)}(\Spec(F),\Z/p^r\Z)
        \]
        is an isomorphism. 
        \end{enumerate}
\end{thm}
This result may be regarded as an analogue of Akhtar’s theorem on higher Chow groups. Specifically, the boundary maps on Milnor $K$-groups define the following complex:
\begin{equation}
\label{eq:MKC}
    \vcenter{
\xymatrix@C=5mm@R=-2mm{
 \cdots \ar[r]^-{\partial} & \displaystyle\bigoplus_{x\in X_{j}}K_{s+j}^M(F(x)) \ar[r]^-{\partial} &\cdots  
 \ar[r]^-{\partial}& \displaystyle\bigoplus_{x\in X_{1}}K_{s+1}^M(F(x)) \ar[r]^-{\partial}  & \displaystyle\bigoplus_{x\in X_{0}}K_{s}^M(F(x)). \\
 & \mbox{\small degree $j$} & & \mbox{\small degree $1$} & \mbox{\small degree $0$} 
}}
\end{equation}
Applying $\otimes_{\Z} \Z/p^r\Z$, 
the Bloch-Gabber-Kato theorem \eqref{eq:BKG} gives an isomorphism 
\[
K_{s+j}^M(F(x))\otimes_{\Z}\Z/p^r\Z \xrightarrow{\simeq} H^{s+j}(F(x),\Z/p^r\Z(s+j))
\]
for each $x\in X_j$. 
Consequently, the complex $KC_{\bullet}^{(s)}(X,\Z/p^r\Z)$ defined in \eqref{eq:KC} can be obtained by shifting the degree by one from the sequence \eqref{eq:MKC}. 
It is known that the $0$-th homology group of the complex \eqref{eq:MKC} is isomorphic to  
a higher Chow group of $X$ as follows (\cite[Thm.~3]{Kat86b}, \cite[Thm.~5.5]{Akh04}):
\[
  \Coker\left(\partial\colon \bigoplus_{x\in X_1}K_{s+1}^M(F(x))\to \bigoplus_{x\in X_0}K_{s}^M(F(x))\right)
  \simeq CH^{d+s}(X,s), 
\]
where $d = \dim(X)$. 
Akhtar’s results are summarized in the following theorem:

\begin{thm}[{\cite[Thm.~1.2 and Thm.~1.3]{Akh04c}, \cite[Cor.~7.2]{Akh04}}]
\label{thm:Akh}
Let $F$ be a field of characteristic $p>0$ 
and $X$ a projective smooth variety over $F$ of $d = \dim(X)$.

\begin{enumerate}
\item 
    If $F$ is a finite field, then 
    we have  $CH^{d+s}(X,s) = 0$ for $s\ge 2$, and 
    the induced map 
    \[ 
       f_{\ast}\colon CH^{d+1}(X,1) \to CH^{1}(\Spec(F),1)
    \]
    is an isomorphism.
    \item 
    If $F$ is a global field, then 
    $CH^{d+2}(X,2)$ is a torsion group and $CH^{d+s}(X,s)=0$ for $s\ge 3$.
    \end{enumerate}
\end{thm}
The most significant application of this theorem is 
the finiteness of the higher Chow group $CH^{d+1}(X,1)$ of $X$ defined over a finite field $F$, 
achieved through the isomorphisms: 
\[
CH^{d+1}(X,1) \xrightarrow[f_{\ast}]{\simeq} CH^1(F,1)  \simeq F^{\times}.
\]
In a similar manner,
our main theorem (\autoref{thm:main_intro}), 
provides insights into the structure of the Kato homology group $KH_0^{(s)}(X,\Z/p^r\Z)$. 
In the case where $F$ is a finite field, we obtain 
\[
KH_0^{(0)}(X,\Z/p^r\Z) \xrightarrow[f_{\ast}]{\simeq} KH_0^{(0)}(\Spec(F),\Z/p^r\Z) =H^1(F,\Z/p^r\Z) \simeq \Z/p^r\Z.
\]
This is a part of the Kato conjecture (see \autoref{conj:Kato}), 
a result proved by Jannsen and Saito (\cite[Thm.~0.3]{JS09}). 
For the case of a global field $F$, 
Kato proposed a conjecture in \cite{Kat86} that there is a short exact sequence 
\begin{equation}
\label{seq:KH0}    
0\to KH_0^{(1)}(X,\Z/p^r\Z) \to \bigoplus_v KH_0^{(1)}(X_v,\Z/p^r\Z) \to \Z/p^r\Z \to 0, 
\end{equation}
where $v$ ranges over the set of all places of $F$ and $X_v := X\otimes_F F_v$ for each place $v$ 
(see \autoref{conj:Kato3}). 
\linelabel{com:2}{\red Our theorem \autoref{thm:main_intro} (2) implies that, 
for the product $X$ of curves over $F$,} 
the $0$-th Kato homology groups coincide with 
the $p^r$-torsion subgroup of the Brauer groups: 
\begin{align*}
    KH_0^{(1)}(X,\Z/p^r\Z) &\xrightarrow[f_{\ast}]{\simeq} KH_0^{(1)}(\Spec(F),\Z/p^r\Z) \stackrel{\eqref{eq:HBr}}{\simeq} \Br(F)[p^r],\quad \mbox{and}\\
KH_0^{(1)}(X_v,\Z/p^r\Z) &\xrightarrow[f_{\ast}]{\simeq }KH_0^{(1)}(\Spec(F_v),\Z/p^r\Z) \stackrel{\eqref{eq:HBr}}{\simeq} \Br(F_v)[p^r] 
\end{align*}
for each place $v$ of $F$. 
The exactness of the sequence \eqref{seq:KH0}
follows from the theorem of Hasse-Brauer-Noether, the exactness of the sequence\linelabel{com:b}
\[
0\to \Br(F){\red [p^r]} \to \bigoplus_v \Br(F_v){\red [p^r]} \xrightarrow{\sum_v\mathrm{inv}_v} \Z/p^r \to 0. 
\]

In \autoref{sec:RS},
we will consider expressing the above theorem \autoref{thm:main_intro} in terms of \emph{reciprocity sheaves} 
by comparing the Mackey product and the tensor product in the category of modulus presheaves with transfers (see \autoref{cor:killed-rec}).

\subsection*{Notation}
Throughout this note, 
for an abelian group $G$ and $m\in \Z_{\ge 1}$,  
we write $G[m]$ and $G/m$ for the kernel and cokernel of the multiplication by $m$ on $G$ respectively. 

For a field $F$, we denote by $\ch(F)$ its characteristic.

\subsection*{Acknowledgements} 
The authors thank the referee for their careful reading and for the many valuable suggestions that improved our manuscript.
The first author was supported by JSPS KAKENHI Grant Number 24K06672.
The second author was supported by JSPS KAKENHI Grant Number 21K03188.

\section{Mackey products}
\label{sec:Mackey}
In this section, we recall the notion of Mackey product following \cite[Sect.~5]{Kah92a} (see also \cite[Rem.~1.3.3]{IR17}, \cite[Sect.~3]{RS00}).

\subsection*{Mackey product}
A \textbf{Mackey functor} $\cM$ over a field $F$ 
(a cohomological finite Mackey functor over $F$ in the sense of \cite{Kah92a}) is a covariant functor $\cM$ 
from the category of field extensions of $F$ to the category of abelian groups equipped with a contravariant structure for finite extensions over $F$ 
    satisfying some conditions (see \cite{Kah92a}, or \cite{Hir24} for the precise definition). 
The category of Mackey functors forms a Grothendieck abelian category (cf.~\cite[Appendix A]{KY13}) 
and hence any morphism of Mackey functor $f\colon \cM\to \cN$, 
that is, a natural transformation, 
gives the image $\operatorname{Im}(f)$, the cokernel $\Coker(f)$ and so on.

\begin{ex}\label{ex:MFexs}
	\begin{enumerate}
	\item 
	For any endomorphism $f\colon \cM\to \cM$ of a Mackey functor $\cM$, 
	we denote the cokernel by 
	$\cM/f := \Coker(f)$. 
	This Mackey functor is given by 
	\[
	(\cM/f)(E) = \Coker\left(f(E)\colon \cM(E)\to \cM(E)\right).
	\]

	\item 
	A commutative algebraic group $G$ over $F$ induces a Mackey functor by defining $E/F\mapsto G(E)$ for any field extension $E/F$
	 (cf.\ \cite[(1.3)]{Som90}, \cite[Prop.~2.2.2]{IR17}). 
	In particular, 
	the multiplicative group $\Gm$ is a Mackey functor given by 
	$\Gm(E) = E^{\times}$ 
	for any field extension $E/F$. 
	The translation maps are the norm 
	$N_{E'/E}\colon (E')^{\times} \to E^{\times}$ and the inclusion $E^{\times} \to (E')^{\times}$. 
	For the additive group $\Ga$, the translation maps are 
	the trace map $\Tr_{E'/E}:E'\to E$ and the inclusion $E\hookrightarrow E'$. 
	\end{enumerate}
\end{ex}

\begin{dfn}\label{def:otimesM}
    For Mackey functors $\cM_1,\ldots, \cM_n$ over $F$, 
    the \textbf{Mackey product} $\cM_1\otimesM \cdots \otimesM \cM_n$ is defined as follows: For any field extension $F'/F$, 
    \begin{equation}\label{eq:otimesM}
		(\cM_1 \otimesM \cdots \otimesM \cM_n ) (F') := 
		\left.\left(\bigoplus_{E/F':\,\mathrm{finite}} \cM_1(E) \otimes_{\Z} \cdots \otimes_{\Z} \cM_n(E)\right)\middle/\ (\textbf{PF}),\right. 
    \end{equation}
    where (\textbf{PF}) stands for the subgroup generated 
	by elements of the following form: 
	For a finite field extensions $F' \subset E \subset E'$, 
	\[
		 x_1 \otimes \cdots\otimes \tr_{E'/E}(\xi_{i_0}) \otimes \cdots \otimes x_{n} - \res_{E'/E}(x_1) \otimes \cdots \otimes \xi_{i_0} \otimes \cdots \otimes \res_{E'/E}(x_n) \quad 
	\]
	for $\xi_{i_0} \in \cM_{i_0}(E')$ and $x_{i} \in \cM_i(E)$ for $i\neq i_0$.
\end{dfn}

For the Mackey product  $\cM_1\otimesM \cdots \otimesM \cM_n$,  
we write $\set{x_1,\ldots ,x_n}_{E/F'}$ 
for the image of $x_1 \otimes \cdots \otimes x_n \in \cM_1(E) \otimes_{\Z} \cdots \otimes_{\Z} \cM_n(E)$ in the product
$(\cM_1\otimesM \cdots \otimesM \cM_n)(F')$. 
The subgroup (\textbf{PF}) gives a relation of the form 
\begin{equation}
    \label{eq:pf}
\set{x_1, \ldots, \tr_{E'/E}(\xi_{i_0}), \ldots,x_{n}}_{E/F'} = \set{x_1,\ldots, \xi_{i_0},\ldots,x_n}_{E'/F'}
\end{equation}
for $\xi_{i_0} \in \cM_{i_0}(E')$ and $x_{i} \in \cM_i(E)$ for $i\neq i_0$
omitting the restriction maps $\res_{E'/E}$ in the right.
The above equation \eqref{eq:pf} is referred to as \textbf{the projection formula}.

The product $\cM_1\otimesM \cdots \otimesM \cM_n$ is a Mackey functor 
with transfer maps defined as follows:
For a field extension $E/F$, 
the restriction 
$\res_{E/F}\colon (\cM_1\otimesM \cdots \otimesM \cM_n)(F) \to (\cM_1\otimesM \cdots \otimesM \cM_n)(E)$ is defined by 
\[
 	\res_{E/F}(\set{x_1,\ldots ,x_n}_{F'/F})= \sum_{i=1}^n e_i\set{\res_{E_i'/E}(x_1), \ldots , \res_{E_i'/E}(x_n)}_{E_i'/E}, 
\]
where $E\otimes_FF' = \bigoplus_{i=1}^n A_i$ for some local Artinian algebras $A_i$  
 of dimension $e_i$ over the residue field $E_i' := A_i/\mathfrak{m}_{A_i}$.
When $E/F$ is finite, 
the transfer map
$\tr_{E/F}\colon  (\cM_1\otimesM \cdots \otimesM \cM_n)(E) \to (\cM_1\otimesM \cdots \otimesM \cM_n)(F)$  
is given by 
\[
	\tr_{E/F}(\set{x_1,\ldots ,x_n}_{E'/E}) = \set{x_1,\ldots ,x_n}_{E'/F}.
\]
The product $\otimesM$ gives  
a tensor structure in the category of Mackey functors. 

\subsection*{Somekawa \texorpdfstring{$K$}{K}-groups}
For semi-abelian varieties $G_1, \ldots, G_n$ over $F$, 
the \textbf{Somekawa $K$-group} $K(F;G_1,\ldots,G_n)$
attached to $G_1, \ldots, G_n$ is the quotient of 
the Mackey product $(G_1\otimesM \cdots \otimesM  G_n)(F)$ 
by the elements coming from \emph{the Weil reciprocity law}
(see \cite{Som90} for the precise definition). 
By definition, 
there is a surjective homomorphism,
$(G_1\otimesM \cdots \otimesM G_n)(F)\twoheadrightarrow K(F;G_1,\ldots, G_n)$.
The elements of $K(F;G_1,\ldots, G_n)$ will also be denoted as linear combinations of symbols of the form 
$\set{x_1,\ldots, x_n}_{F'/F}$, where $F'/F$ is a finite extension and $x_i\in G_i(F')$ for $i=1,\ldots, n$.

\begin{lem}\label{lem:product}
Let $G_1,\ldots, G_n$ be semi-abelian varieties over $F$ for $n\ge 3$. 
There is a surjective homomorphism 
\[
    \left(K(-;G_1,G_2)\otimesM G_3\otimesM \cdots \otimesM G_n\right)(F) \twoheadrightarrow K(F;G_1,\ldots, G_n),
\]
where $K(-;G_1,G_2)$ is the Mackey functor defined by 
$F'/F \mapsto K(F';G_1,G_2)$. 
\end{lem}
\begin{proof}
As the Mackey product forms a tensor structure in the category of Mackey products, 
we have 
\[
\xymatrix{
\left((G_1\otimesM G_2) \otimesM G_3 \otimesM\cdots \otimesM G_n \right)(F)\ar@{->>}[d] \ar[r]^-{\simeq} & \left(G_1 \otimesM\cdots \otimesM G_n \right)(F)\ar@{->>}[d] \\
\left(K(-;G_1,G_2)\otimesM G_3\otimesM \cdots \otimesM G_n\right)(F)  \ar@{-->}[r] & K(F;G_1,\ldots, G_n)
}
\]
The Weil reciprocity in $K(-;G_1,G_2)$ is vanished in $K(F;G_1,\ldots,G_n)$ 
(see \cite[Proof of Prop.~3.1]{GL24}) the dotted arrow in the above diagram is defined well. 
\end{proof}

\subsection*{Higher Chow groups}

The boundary maps on Milnor $K$-groups give a complex (cf.~\cite[Sect.~1]{Kat86})
\begin{equation}
\label{eq:MKC2}
    \vcenter{
\xymatrix@C=5mm@R=-2mm{
 \cdots \ar[r]^-{\partial} & \displaystyle\bigoplus_{x\in X_{j}}K_{s+j}^M(F(x)) \ar[r]^-{\partial} &\cdots  
 \ar[r]^-{\partial}& \displaystyle\bigoplus_{x\in X_{1}}K_{s+1}^M(F(x)) \ar[r]^-{\partial}  & \displaystyle\bigoplus_{x\in X_{0}}K_{s}^M(F(x)). \\
 & \mbox{\small degree $j$} & & \mbox{\small degree $1$} & \mbox{\small degree $0$} 
}}
\end{equation}
The $0$-th homology group of the above complex \eqref{eq:MKC2} is isomorphic to  
the higher Chow group of $X$ (\cite[Thm.~3]{Kat86b}, \cite[Thm.~5.5]{Akh04}):
\begin{equation}
\label{eq:MK-Chow}    
  \Coker\left(\partial\colon \bigoplus_{x\in X_1}K_{s+1}^M(F(x))\to \bigoplus_{x\in X_0}K_{s}^M(F(x))\right)
  \simeq CH^{d+s}(X,s), 
\end{equation}
where $d = \dim(X)$. 
Here, we slightly recall the construction of the isomorphism above \eqref{eq:MK-Chow}. 
The higher Chow group $CH^{d+s}(X,s)$ is generated by 
closed points of $X\times \cube^s$, 
where $\cube^s = (\mathbb{P}^1\smallsetminus \set{1})^s$. 
By considering the projections $X\times \cube^s\to X$ and $X\times \cube^s \to \cube$, a closed point $P$ of $X\times \cube^s$ 
is determined by a closed point $x \in X$ and 
$b_1,\ldots,b_s \in F(P)$. 
We write 
\begin{equation}\label{eq:P}    
P = x\times b_1\times \cdots \times b_s
\end{equation} 
for simplicity.
The class $[P]$ in $CH^{d+s}(X,s)$ 
corresponds to 
$N_{F(P)/F(x)}(\set{b_1,\ldots,b_s})$ at $x \in X_0$ in 
$\bigoplus_{x\in X_0}K_s^M(F(x))$. 
Conversely, 
for a symbol $\set{b_1,\ldots,b_s} \in K_s^M(F(x))$, 
identifying the isomorphism $F(x)^{\times} \simeq CH^{1}(F(x),1)$, 
the intersection product gives 
$(j_{x})_{\ast}([x]\cdot b_1\cdot \cdots \cdot b_s) \in CH^{d+s}(X,s)$, 
where $j_x \colon X_{F(x)}\to X$. 

The structure map $f\colon X\to \Spec(F)$ induces a map 
\begin{equation}
    \label{eq:fCH}
    f_{\ast}^{CH}\colon CH^{d+s}(X,s) \to CH^s(F,s) \simeq K_s^M(F),  
\end{equation}
where the last isomorphism is the Nesterenko-Suslin/Totaro theorem (\cite{NS89}, \cite{Tot92}). 
The kernel is denoted by 
\[
A^{d+s}(X,s) := \Ker(f_{\ast}^{CH}).
\]

We consider the Mackey functor $\underline{CH}^{d+s}(X,s)$ 
defined by the higher Chow group $\underline{CH}^{d+s}(X,s)(E) = CH^{d+s}(X_E,s)$ for each extension $E/F$, 
where $X_E := X\otimes_F E$.
The transfer maps are given by pull-back and the push-forward map 
along $j:\Spec(E) \to \Spec(F)$  (see, for example \cite[Sect.~4]{Akh04}). 
The map $f_{\ast}^{CH}$ induces a map
$f_{\ast}^{CH}\colon \underline{CH}^{d+s}(X,s) \to K_s^M$ 
 of Mackey functors 
and its kernel is denoted by $\underline{A}^{d+s}(X,s)$. 
\begin{thm}[{\cite[Thm.~6.1]{Akh04}}]\label{thm:Akh2}
There is an isomorphism 
\[
K(F;\underline{CH}_0(X),\overbrace{\Gm,\ldots,\Gm}^s) \xrightarrow{\simeq} CH^{d+s}(X,s),
\]
for $s\ge 0$, where  $\underline{CH}_0(X) = \underline{CH}^{d}(X,0)$ and 
the left hand side is the Somekawa type $K$-group 
which is a quotient of the Mackey product 
$\Big(\underline{CH}_0(X)\otimesM \Gm^{\otimesM s}\Big)(F)$. 
\end{thm}

\section{Extended differential symbol map}
\label{sec:diff}
Let $F$ be a field of $\ch(F) =p>0$.

\subsection*{Differential symbol map}
For a field extension $E/F$, 
the $p$-th power Frobenius $\phi\colon E\to E; a\mapsto a^p$
gives a homomorphism 
\begin{equation}
	\label{eq:wp}
\wp = W_r(\phi)-\Id: W_r(E) \to W_r(E)
\end{equation}
for any $r\ge 1$.
For a Mackey functor $\cM$ over $F$, 
the map $\wp$ induces an endomorphism of the Mackey product 
$(W_r \otimesM \cM) (F)$ 
by setting $\set{\ba, x}_{E/F} \mapsto \set{\wp(\ba),x}_{E/F}$. 
In the following, this endomorphism is denoted by $\wp := \wp\otimes\Id$ and we put
\[
(W_r \otimesM \cM) (F)/\wp := \Coker\left(\wp\colon (W_r \otimesM \cM) (F) \to (W_r \otimesM \cM) (F)\right).
\]

\begin{lem}
\label{lem:normGa}
	Let $\cM$ and $\cN$ be Mackey functors over $F$. 
	
	\begin{enumerate}
		\item 	
	For a finite separable field extension $F'/F$, 
	the transfer map $\tr_{F'/F}\colon (W_r \otimesM \cM)(F')\to (W_r\otimesM \cM)(F)$ 
	is surjective for any $r \ge 1$. 
        \item 
        For a morphism $\cM\to \cN$, 
        if $(\Ga\otimesM \cM)(F)/\wp \to (\Ga\otimesM \cN)(F)/\wp$ is surjective, then so is $(W_r\otimesM \cM) (F)/\wp\to (W_r\otimesM \cN)(F)/\wp$ for any $r\ge 1$.
	\item 
	If we have $(\Ga\otimesM \cM)(F)/\wp  = 0$, then $(W_r\otimesM \cM)(F)/\wp =0$ for any $r\ge 1$. 
	\end{enumerate}
\end{lem}
\begin{proof}
   The short exact sequence 
    $0\to W_{r}(F) \to W_{r+1}(F) \to \Ga(F) \to 0$ induces a right exact sequence 
    \[
     (W_{r}\otimesM \cM)(F)\to  (W_{r+1}\otimesM \cM)(F) \to (\Ga\otimesM\cM)(F)\to 0.
    \]

\noindent 
(1) 
    Take any symbol $\set{a,x}_{E/F}$ in $(\Ga\otimesM \cM)(F)$. 
    Since the field extension $EF'/E$ is separable, 
    the trace map $\Tr_{EF'/E}\colon EF'\to E$ 
     is surjective. 
    There exists $a'\in EF'$ such that $\Tr_{EF'/E}(a') = a$. 
    By the projection formula (\textbf{PF}), we have 
    \[
        \set{a,x}_{E/F} = \set{\Tr_{EF'/E}(a'), x}_{E/F} = \set{a', x}_{EF'/F} = 
        \tr_{F'/F}(\set{a',x}_{EF'/F'}). 
    \]
    The transfer map $\tr_{F'/F}\colon (\Ga \otimesM \cM)(F')\to (\Ga\otimesM \cM)(F)$ is {\red therefore} surjective.\linelabel{com:3}
    Now, we assume the transfer map $\tr_{F'/F}\colon (W_{r}\otimesM \cM)(F') \to (W_{r}\otimesM \cM)(F)$ is surjective for $r \ge 1$. We have a commutative diagram with exact rows: 
    \[
    \xymatrix{
    (W_{r}\otimesM \cM)(F')\ar@{->>}[d]^{\tr_{F'/F}}\ar[r] & (W_{r+1}\otimesM \cM)(F') \ar[r] \ar[d]^{\tr_{F'/F}}&  (\Ga\otimesM\cM)(F')\ar@{->>}[d]^{\tr_{F'/F}} \ar[r] &  0\\
    (W_{r}\otimesM \cM)(F)\ar[r] & (W_{r+1}\otimesM \cM)(F) \ar[r]&  (\Ga\otimesM\cM)(F) \ar[r] &  0.
    }
    \]
    From this diagram, the map $\tr_{F'/F}\colon (W_{r+1}\otimesM \cM)(F')\to (W_{r+1}\otimesM \cM)(F)$ is surjective.

    \noindent
    (2) 
   Consider the following commutative diagram with exact rows:
    \[
    \xymatrix{
    (W_{r}\otimesM \cM)(F)/\wp\ar[d]\ar[r] & (W_{r+1}\otimesM \cM)(F)/\wp \ar[r] \ar[d]&  (\Ga\otimesM\cM)(F)/\wp\ar[d] \ar[r] &  0\\
    (W_{r}\otimesM \cN)(F)/\wp\ar[r] & (W_{r+1}\otimesM \cN)(F)/\wp \ar[r]&  (\Ga\otimesM\cN)(F)/\wp \ar[r] &  0.
    }
    \]
    The induction on $r$ and the assumption implies the middle vertical map in the above diagram is surjective. 
    The assertion (3) also follows from the induction on $r$ and the assumption 
    $(\Ga \otimesM \cM)(F)/\wp = 0$. 
\end{proof}

There is a natural homomorphism 
$\dlog: F^{\times} \to W_r\Omega_F^1; b\mapsto \dlog[b] := \mathrm{d}[b]/[b]$, 
where $[b] = (b,0,\ldots, 0)$. 
The subgroup of $W_r\Omega_F^{n}$ generated by 
$\dlog [b_1] \cdots \dlog [b_n]$ for $b_1,\ldots,b_n\in F^{\times}$ 
is denoted by $W_r\Omega_{F,\log}^n$. 
We put 
\begin{equation}\label{def:H}
	H^n(F,\Z/p^r(m)) = H^{n-m}(F,W_r\Omega_{F^{\sep},\log}^m). 
\end{equation}

\begin{thm}[{Bloch-Gabber-Kato, \cite[Cor.~2.8]{BK86}}]
	\label{thm:BKG}
	For any $r \in \Z_{>0}$ and $n \in \Z_{\ge 0}$, the \textbf{differential symbol map} 
	\[
	s_{F,p^r}^n:K_n^M(F)/p^r \xrightarrow{\simeq} H^n(F,\Z/p^r(n)) = W_r\Omega_{F,\log}^n; \set{b_1,\ldots,b_n}\mapsto \dlog[b_1]\cdots\dlog[b_n]
	\]
	is an isomorphism.
\end{thm}

\subsection*{Extended differential symbol map}
Following \cite{Kat80}, 
we put 
\begin{equation}
\label{def:Hpr}
	H^{n+1}_{p^r}(F) := H^{n+1}(F,\Z/p^r(n)) = H^1(F,W_r\Omega_{\log}^n).
\end{equation}
By \cite[Sect.~1, Lem.~2]{CSS83}, 
there is a short exact sequence 
\[
0 \to W_r\Omega^{n}_{\log} \to W_r\Omega^{n}\to W_r\Omega^{n}/\mathrm{d}V^{r-1}\Omega^{n-1}\to 0
\]
of \'etale sheaves on $\Spec(F)$. 
The map $\wp:W_r(E)\to W_r(E)$ defined in \eqref{eq:wp} induces the following 
short exact sequence: 
\begin{equation}\label{eq:strH^n}
	0\to H^{n}(F,\Z/p^r(n))\to W_r\Omega_F^{n}\xrightarrow{\wp} W_r\Omega_F^{n}/\mathrm{d}V^{r-1}\Omega_F^{n-1}\to H^{n+1}_{p^r}(F)\to 0 , 
\end{equation}
where the middle $\wp$ is given by $\wp(\mathbf{a}\dlog[b_1]\cdots \dlog [b_{n-1}]) = \wp(\mathbf{a}) \dlog[b_1]\cdots \dlog[b_{n-1}]$ (\cite[Sect.~A.2]{JSS14}). 
Moreover, the following explicit description of $H^{n+1}_{p^r}(F)$ is known.

\begin{thm}[{\cite{Kat82c}, \cite[Sect.~3]{Kat80}, see also \cite[Prop.~A.2.4]{JSS14}}]
\label{thm:Kato}
	The natural homomorphism 
	$\ba\otimes b_1\otimes \cdots \otimes b_{n} \mapsto \ba\dlog[b_1]\cdots \dlog[b_{n}]$ 
	induces an isomorphism 
	\[
	\left(W_r(F)\otimes_\Z (F^{\times})^{\otimes_\Z n}\right)/J \xrightarrow{\simeq} H^{n+1}_{p^r}(F),
	\]
	where 
	$J$ is the subgroup generated by all elements of the form:
	\begin{enumerate}[label=$(\mathrm{\alph*})$]
		\item $\ba\otimes b_1\otimes \cdots \otimes b_{n}$ with $b_i = b_j$ for some $i\neq j$.
		\item $(0,\ldots, 0, a, 0,\ldots ,0)\otimes a\otimes b_2\otimes \cdots \otimes b_{n}$ 
		for some $a \in k$.
		\item $\wp(\mathbf{a})\otimes b_1\otimes \cdots  \otimes b_{n}$.
	\end{enumerate}
\end{thm}

Through the isomorphism above, \linelabel{com:4} 
we define {\red the} symbol $\AS{\ba,b_1,\ldots ,b_{n}}_F$ as {\red the} element in $H^{n+1}_{p^r}(F)$ 
{\red given} by the image of the residue class represented by $\ba \otimes b_1\otimes \cdots \otimes b_{n}\in W_r(F)\otimes_\Z (F^{\times})^{\otimes_\Z n}$.

We recall that, for any finite field extension $E/F$, 
the trace map $\Tr_{E/F}\colon W_r\Omega^n_E\to W_r\Omega^n_F$ 
defined in \cite[Thm.~2.6]{Rue07} (\cite[Thm.~A.2]{KP21}) 
induces a map 
\begin{equation}
    \label{def:trace}
    \Tr_{E/F}\colon H^{n+1}_{p^r}(E) \to H_{p^r}^{n+1}(F).
\end{equation}

\begin{thm}[{{\red \cite[Thm.~1.1]{Hir24}}}]\label{thm:H24}
    For $r\ge 1$ and $n\ge 0$, 
    there is an isomorphism
    \[
	\tilde{s}_{F,p^r}^{\,n}\colon \Big(W_r \otimesM \Gm^{\otimesM n}\Big) (F)/\wp \xrightarrow{\simeq} H^{n+1}_{p^r}(F); \set{\ba,b_1,\ldots,b_n}_{E/F}\mapsto \Tr_{E/F}(\AS{\ba,b_1,\ldots,b_n}_{E}).
    \]
\end{thm}

By Kahn's theorem \cite[Thm.~4.5]{Hir24}, 
the natural surjective map $\Gm^{\otimesM n}(F) \twoheadrightarrow K^M_n(F)$ 
induces an isomorphism $\Gm^{\otimesM n}(F)/p^r \simeq K^M_n(F)/p^r$. 
For the unit $\mathbf{1}$ of $W_r$, 
the natural map 
\[
\Gm^{\otimesM n}(F) \to \Big(W_r\otimesM \Gm^{\otimesM n}\Big)(F); \set{b_1,\ldots,b_n}_{F'/F} \mapsto 
\set{\mathbf{1},b_1,\ldots,b_n}_{F'/F}
\]
makes the following diagram commutative: 
\begin{equation}
\label{diag:dsm}
\vcenter{
\xymatrix{
K_n^M(F)/p^r \ar[r]^-{s_{F,p^r}^n}_-{\simeq}\ar[d]_{\set{\mathbf{1},-}_{F/F}} &  H^n(F,\Z/p^r(n)) \ar[d]^{\AS{\mathbf{1},-}_F} \\ 
 \Big(W_r \otimesM \Gm^{\otimesM n}\Big) (F)/\wp \ar[r]_-{\simeq}^-{\tilde{s}_{F,p^r}^{\,n}} & H^{n+1}_{p^r}(F),
}}
\end{equation}
where the right vertical map is given by 
$\dlog[b_1]\cdots \dlog[b_n] \mapsto \dlog[b_1]\cdots \dlog[b_n] = \AS{\mathbf{1},b_1,\ldots,b_n}_F$.

\subsection*{Reciprocity functors}
{\red 
The \textbf{inverse Cartier operator}\linelabel{com:6'} 
\[
C^{-1}\colon \Omega_F^n \to \Omega_F^n/\mathrm{d}\Omega_F^{n-1}
\]
is given by 
$C^{-1}\left(a\dlog b_1\cdots \dlog b_n\right) = a^p \dlog b_1\cdots \dlog b_n$. 
The map $\wp$ defined in \eqref{eq:strH^n} for $r=1$ is written as}
\begin{equation}\label{eq:def-wp}
 {\red \wp = C^{-1} - \Id \colon \Omega^n_F \to \Omega_F^n/\mathrm{d}\Omega_F^{n-1}.}
\end{equation}
We define a subgroup $B_\infty:=\bigcup_i B_i\subset \Omega_F^n$ as follows;
$B_1=\mathrm{d}\Omega_F^{n-1}$ and for $i\geq 2$, we define $B_i\subset \Omega_F^n$ 
as the preimage of $C^{-1}(B_{i-1})$ 
along the natural surjection $\Omega_F^n \to \Omega_F^n/B_1$. 
This filtration $(B_i)$ satisfies  
\begin{equation}\label{eq:invCar}    
    C^{-1}\colon B_i \xrightarrow{\simeq} B_{i+1}/B_1 
\end{equation}
(cf.~\cite[Sect.~2.2]{Ill79}). 
{\red \linelabel{com:6}Therefore, the map $\wp = C^{-1} - \Id$ (cf.~\eqref{eq:def-wp}) 
maps $B_i$ to $B_{i+1}/B_1$ 
and hence we have $\wp\colon B_\infty \to B_{\infty}/B_1$. 
There is the following commutative diagram with exact rows:}
\begin{equation}\label{diag:Binf}  
\vcenter{
\xymatrix{
0\ar[r] &B_\infty \ar[d]^{\wp}\ar[r] & \Omega_F^{n}\ar[r]\ar[d]^{\wp} & \Omega_F^{n}/B_\infty\ar@{-->}[d]^{\wp}\ar[r] &0\\
0\ar[r] &B_\infty/B_1 \ar[r] & \Omega_F^{n}/B_1\ar[r] & \Omega_F^{n}/B_\infty \ar[r] &0.
}}
\end{equation}

\begin{cor}\label{cor:Binf}
We have an isomorphism
\[
     \Big(\Ga \otimesM \Gm^{\otimesM n} \Big) (F)/\wp \simeq \Coker\left( \wp\colon \Omega_F^n/B_{\infty} \to \Omega_F^n/B_{\infty}\right).
\]
\end{cor}
\begin{proof}
By \autoref{thm:H24}, it is enough to show that $H_{p}^{n+1}(F)\simeq \Coker( \wp\colon \Omega_F^n/B_{\infty} \to \Omega_F^n/B_{\infty})$.
Since the cokernel of the middle vertical map in the diagram \eqref{diag:Binf} is known to be isomorphic to $H_{p}^{n+1}(F)$ by \eqref{eq:strH^n}, our task now is to show that the left vertical map $\wp\colon B_{\infty}\to B_{\infty}/B_1$ is surjective. 

We prove this by induction on $i$ for $B_i$. 
Take any $x\in B_\infty$ and if $x\in B_2$ then 
$C(\overline{x}) \in B_1$, 
where $\overline{x}$ is the residue class of $x$ in $B_2/B_1$ 
and $C$ is the Cartier operator $C\colon B_2/B_1\xrightarrow{\simeq } B_1$ (cf.~\eqref{eq:invCar}), 
we have 
\[
\wp(C(\overline{x})) = C^{-1}(C(\overline{x})) - C(\overline{x}) =  \overline{x}\quad \mbox{in $B_2/B_1$}.
\]
The residue class of $x$ in $B_\infty/B_1$ belongs to $\mathrm{Im}(\wp)$.
We now suppose that
$B_i/B_1$ is in the image of $\wp$ for $i>1$. 
For $x\in B_{i+1}$, 
the {\red Cartier} operator $C\colon B_{i}\xrightarrow{\simeq}B_{i+1}/B_1$ 
gives 
\[
\wp(C(\overline{x})) = \overline{x} - C(\overline{x})
\]
and by induction on $i$ the residue class of $x$ in $B_\infty/B_1$ belongs to $\mathrm{Im}(\wp)$.
\end{proof}
Now, we assume that the field $F$ is a finitely generated field over a perfect field. 
The product $T(\Ga,\Gm^{\times n})(F)$ as the reciprocity functors (\cite[Sect.~4]{IR17}) 
is a quotient of the Mackey product $(\Ga\otimesM \Gm^{\otimesM n})(F)$, and  
one can define an endomorphism 
\[
\wp := \wp \otimes \Id \colon T(\Ga,\Gm^{\times n})(F)\to  T(\Ga,\Gm^{\times n})(F)
\] 
(\cite[Sect.~5.4]{IR17}). 
By \cite[Cor.~5.4.12]{IR17} 
there is a surjective homomorphism $\Omega_F^n/B_{\infty} \twoheadrightarrow T(\Ga,\Gm^{\times n})(F)$.
The above corollary 
gives a surjective homomorphism 
\[
 \Big(\Ga \otimesM \Gm^{\otimesM n} \Big) (F)/\wp \twoheadrightarrow \Coker\left(\wp\colon T(\Ga,\Gm^{\times n})(F)\to T(\Ga,\Gm^{\times n})(F)\right).
\]

\section{Kato homology}
\label{sec:KH}
Let $F$ be a field of  $\ch(F) = p>0$ 
with $[F:F^p]\le p^s$ for some $s$. 
Throughout this section, 
let $X$ be a \emph{projective smooth and geometrically irreducible scheme} over $F$ 
of $d = \dim(X)$.
We denote by $X_j$ the set of points $x$ in $X$ with $\dim(\overline{\set{x}}) = j$ and 
by $F(x)$ the residue field at $x$.  

\subsection*{Kato homology}
We have the following homological complex $KC^{(s)}_{\bullet}(X,\Z/p^r)$ of Bloch-Ogus type (\cite[Prop.~1.7]{Kat86}): 
\[
\xymatrix@C=5mm@R=0mm{
 \cdots \ar[r]^-{\partial} & \displaystyle\bigoplus_{x\in X_{j}}H^{s+j+1}_{p^r}(F(x)) \ar[r]^-{\partial} &\cdots  
 \ar[r]^-{\partial}& \displaystyle\bigoplus_{x\in X_{1}}H^{s+2}_{p^r}(F(x)) \ar[r]^-{\partial}  & \displaystyle\bigoplus_{x\in X_{0}}H^{s+1}_{p^r}(F(x)). \\
 & \mbox{\small degree $j$} & & \mbox{\small degree $1$} & \mbox{\small degree $0$} 
}
\]
The \textbf{Kato homology group} of $X$ (with coefficients in $\Z/p^r$) is defined 
to be the homology group 
\begin{equation}
\label{def:KH}
KH_j^{(s)}(X,\Z/p^r) := H_j(KC_{\bullet}^{(s)}(X,\Z/p^r))
\end{equation}
of the above complex $KC_\bullet^{(s)}(X,\Z/p^r)$. 
By \cite[Sect.~2]{JS09}, 
the structure map $f\colon X\to \Spec(F)$ induces a map 
of complexes $KC_{\bullet}^{(s)}(X,\Z/p^r) \to KC_{\bullet}^{(s)}(\Spec(F),\Z/p^r)$ 
and hence a complex 
\begin{equation}\label{eq:RecCpx}
\bigoplus_{x\in X_1} H_{p^r}^{s+2}(F(x)) \xrightarrow{\partial} \bigoplus_{x\in X_0} H_{p^r}^{s+1}(F(x))  \xrightarrow{\Cor} H_{p^r}^{s+1}(F).
\end{equation}
Here, the map $\Cor$ is defined by the corestrictions $\Cor_{F(x)/F}\colon H_{p^r}^{s+1}(F(x)) \to H_{p^r}^{s+1}(F)$  for $x\in X_0$ (cf.~\cite[Appendix A]{JSS14}). 
By taking the homology group, there is a map 
\begin{equation}\label{eq:fast}    
f_{\ast}^{KH}\colon KH_0^{(s)}(X,\Z/p^r) \to KH_0^{(s)}(\Spec(F),\Z/p^r) = H_{p^r}^{s+1}(F). 
\end{equation}

The map $f_{\ast}^{KH}$ is surjective in general (cf.~\autoref{lem:varphi} below). 
The kernel of the map $f_{\ast}^{KH}$ is the homology group 
of the complex \eqref{eq:RecCpx}. 
The  
diagram \eqref{diag:dsm} induces a map 
from the complex \eqref{eq:MKC2} to $KC_{\bullet}^{(s)}(X,\Z/p^r)$
of complexes. The following diagram is commutative:
\begin{equation}
    \label{diag:dlog}
\xymatrix{
CH^{d+s}(X,s)\ar[d]_{\dlog} \ar[r]^-{f_\ast^{CH}} & K_s^M(F)\ar[d]^{\dlog}\\ 
KH_0^{(s)} (X,\Z/p^r) \ar[r]^-{f_{\ast}^{KH}} & H_{p^r}^{s+1}(F). 
}
\end{equation}

\subsection*{Representation theorems}
Following Akhtar's theorem (\autoref{thm:Akh2}), we give a representation theorem of 
the Kato homology group $KH_0^{(s)}(X,\Z/p^r)$. 

\begin{prop}\label{prop:surj}
	For any $r\ge 1$, there is a surjective homomorphism
	 \[
 \psi\colon \Big(W_r\otimesM\underline{CH}^{d+s}(X,s)\Big)(F)/\wp \twoheadrightarrow KH_0^{(s)}(X,\Z/p^r).
 \]
\end{prop}
\begin{proof}
    For any finite field extension $E/F$, 
    the extended differential symbol map $\tilde{s}_{E,p^r}$ 
    (\autoref{thm:H24})
    induces the map 
    \[
    W_r(E)\otimes_{\Z} K_{s}^M(E) \xrightarrow{ \set{-,-}_{E/E}}  \Big(W_r\otimesM K_s^M\Big)(E)/\wp \stackrel{(\clubsuit)}{\simeq} 
    \Big(W_r\otimesM \Gm^{\otimesM s}\Big)(E)/\wp   \xrightarrow[\simeq]{\tilde{s}_{E,p^r}} H_{p^r}^{s+1}(E)
    \]
    which is also denoted by $\tilde{s}_{E,p^r}$. 
    Here, the isomorphism ($\clubsuit$) follows from 
     $\Gm^{\otimesM s}(F)/p^r \simeq K_s^M(F)/p^r$ (\cite[Thm.~4.5]{Hir24}).
    
    For each finite field extension $E/F$, 
    there is a commutative diagram: 
    \begin{equation}\label{diag:psi-2}
        \vcenter{
        \xymatrix@C=15mm{
		\displaystyle\bigoplus_{\eta \in (X_E)_1}W_r(E)\otimes_\Z K_{s+1}^M(E(\eta))\ar[r]^-{\Id \otimes \partial}\ar[d]_-{\tilde{s}_{E(\eta),p^r}^{s+1}} & \displaystyle\bigoplus_{y \in (X_E)_0}W_r(E)\otimes_\Z  K_s^M(E(y))\ar[d]^{\tilde{s}_{E(y),p^r}^{s}} \\
		\displaystyle\bigoplus_{\eta \in (X_E)_1}H^{s+2}_{p^r}(E(\eta))\ar[r]^{\partial}\ar[d]_-{\Cor} & \displaystyle\bigoplus_{y\in (X_E)_0}H^{s+1}_{p^r}(E(y))\ar[d]^-{\Cor}\\
				\displaystyle\bigoplus_{\xi \in X_1}H^{s+2}_{p^r}(F(\xi))\ar[r]^{\partial} & \displaystyle\bigoplus_{x\in X_0}H^{s+1}_{p^r}(F(x)).
		}
		}
    \end{equation}
    The bottom left vertical map $\Cor$ is defined by the corestriction maps as follows: 
    For each $\eta \in (X_E)_1$, let $\xi \in X_1$ be the image of $\eta$ by the natural map $j\colon X_E\to X$  induced from the extension $F\subset E$. 
    We define $\AS{\ba,b_1,\ldots,b_s}_{E(\eta)} \mapsto \Cor_{E(\eta)/F(\xi)}(\AS{\ba,b_1,\ldots,b_s}_{E(\eta)})$.\linelabel{com:7} 
    The bottom right vertical map $\Cor$ is defined similarly. 

   {\red We have a representation of the higher Chow group  \linelabel{com:a} 
    \begin{equation}\label{eq:MK-ChowE}
    \Coker\left(\partial\colon \bigoplus_{\eta\in (X_E)_1}K_{s+1}^M(E(\eta)) \to \bigoplus_{y\in (X_E)_0}K_s^M(E(y))\right) \simeq CH^{d+s}(X_E,s)
    \end{equation}
    using the Milnor $K$-groups referred in \eqref{eq:MK-Chow}.
    }
    By taking the cokernels of the horizontal maps in \eqref{diag:psi-2}, 
    we obtain  
    \[
	\psi_E\colon W_r(E)\otimes_\Z CH^{d+s}(X_E,s)
        \to KH_0^{s}(X_E,\Z/p^r) \xrightarrow{j_{\ast}} KH_0^{(s)}(X,\Z/p^r).
    \]
    Explicitly, 
    take $\mathbf{a} \in W_r(E)$ and $[Q] \in CH^{d+s}(X_E,s)$ {\red \linelabel{com:8}with 
    $Q = y \times b_1\times \cdots \times b_s$ for 
    $y \in (X_E)_0$ and $b_1,\ldots, b_s \in E(Q)^{\times}$ 
    using the convention introduced in \eqref{eq:P}.}
    Putting $\bb = \set{b_1,\ldots, b_s} \in K_s^M(E(Q))$, 
    we have 
    \begin{align*}        
    \psi_E( \mathbf{a} \otimes [Q]) &= \Cor_{E(y)/F}\left(\tilde{s}_{E(y),p^r}^s(\set{\mathbf{a},N_{E(Q)/E(y)}(\bb)}_{E(y)/E(y)})\right)\\
    &\stackrel{(\mathbf{PF})}{=} \Cor_{E(y)/F}\left(\tilde{s}_{E(y),p^r}^s(\set{\res_{E(Q)/E}(\mathbf{a}),\bb}_{E(Q)/E(y)})\right)\\
    &= \Cor_{E(Q)/F}\left( \AS{\mathbf{a}, b_1, \ldots, b_s}_{E(Q)}\right).
    \end{align*}
    \linelabel{com:8'}{\red Since the higher Chow group $CH^{d+s}(X_E,s)$ is generated by elements of the form $[Q]$ for closed points $Q = y\times b_1\times \cdots \times b_s$ as above, 
    we obtain} a homomorphism  
    \begin{equation}\label{eq:psiE}
        \psi = (\psi_E)_E \colon \bigoplus_{E/F} W_r(E)\otimes_\Z CH^{d+s}(X_E,s)\to KH_0^{(s)}(X,\Z/p^r).
    \end{equation}
    We show that the map $\psi$ kills the projection formula (\textbf{PF}) in the definition of the Mackey product. 
    Our task now is to show the following: 
    Take finite extensions $F\subset E \subset E'$ and we denote by $j\colon X_{E'}\to X_E$ the map 
    induced from $E\hookrightarrow E'$.
    \begin{enumerate}[label=(\alph*)]
        \item For any $\ba' \in W_r(E')$ and $[Q] \in CH^{d+s}(X_E,s)$, we have 
        \[
        \psi_{E'}(\ba'\otimes j^{\ast}([Q])) = \psi_E(\Tr_{E'/E}(\ba')\otimes[Q]).
        \]
    
        \item For any $\ba \in W_r(E)$ and $[Q'] \in CH^{d+s}(X_{E'},s)$, we have 
        \[
        \psi_{E'}(\ba \otimes [Q']) = \psi_E(\ba\otimes j_{\ast}([Q'])).
        \] 
    \end{enumerate}

    For the equality (a), 
    take $\ba' \in W_r(E')$, the higher $0$-cycle $[Q] \in CH^{d+s}(X_E,s)$ 
    is 
    represented by $Q  = y\times b_1\times \cdots \times b_s \in (X_E\times \cube^s)_0$ 
    for finite extensions $F\subset E\subset E'$. 
    The pull-back $j^{\ast}([Q]) = (j\times \Id)^{\ast}([Q])$ is given by the pull-back along $j\times \Id \colon X_{E'}\times \cube^s\to X_E\times \cube^s$. 
    We have 
    \begin{equation}\label{eq:jastQ}
    j^{\ast}([Q]) = [j^{-1}(y)\times b_1\times \cdots \times b_s] = 
    \sum_{y'\mid y} e_{y'} [y'\times b_1\times \cdots \times b_s], 
    \end{equation}
    for some $y' \in (X_{E'})_0$ over $y$ and $e_{y'}\in \Z$. 
    The base change gives $E(Q) \otimes_E {E'} \simeq \prod_{y'} A_{y'}$ 
    for some local Artinian algebras $A_{y'}$ of dimension $e_{y'}$ over ${E'}(Q')$ for each 
    $Q' = y'\times b_1\times \cdots\times b_s$ with 
    $y' \in (X_{E'})_0$ over $y$. 
    The very definition of the Mackey functor $W_r$ gives the equality
    \begin{equation}\label{eq:trres}
        \sum_{y'\mid y}\Tr_{E'(Q')/E(Q)}\left(e_{y'}\res_{E'(Q')/E'}(\ba')\right) 
            = \res_{E(Q)/E}\left(\Tr_{E'/E}(\ba')\right).
    \end{equation}
    Putting $\bb=\set{b_1,\ldots,b_s} \in K_s^M(E(Q))$, we have 
    \begin{align*}
    \psi_{E'}(\ba' \otimes j^{\ast}([Q])) 
      &\stackrel{\eqref{eq:jastQ}}{=} \sum_{y'\mid y} e_{y'}  
        \psi_{E'} ( \ba'\otimes [Q']) \\
        &=  \sum_{y'\mid y} e_{y'} \Cor_{E'(Q')/F}\circ \tilde{s}_{E'(Q'),p^r}^s(\set{\ba', \bb}_{E'(Q')/E'(Q')}) \\
        & = \Cor_{E(Q)/F}\left(  \sum_{y'\mid y} e_{y'}  \tilde{s}_{E(Q),p^r}^s(\set{\ba',\bb}_{E'(Q')/E(Q)}) \right)\\
        & \stackrel{(\mathbf{PF})}{=} \Cor_{E(Q)/F}\left(  \sum_{y'\mid y} e_{y'}  \tilde{s}_{E(Q),p^r}^s(\set{\Tr_{E'(Q')/E(Q)}(\ba'),\bb}_{E(Q)/E(Q)}) \right)\\
    &\stackrel{\eqref{eq:trres}}{=}   \Cor_{E(Q)/F}\left(\AS{\Tr_{E'/E}(\mathbf{a}'), b_1,\ldots,  b_s}_{E(Q)}\right) \\
    & = \psi_E(\Tr_{E'/E}(\mathbf{a}')\otimes [Q]).
    \end{align*}

    To show the equality (b), 
    take $\ba \in W_r(E)$ and $[Q'] \in CH^{d+s}(X_{E'},s)$ 
    with $Q' = y'\times b_1'\times  \cdots \times b_s'$ 
    for finite extensions $F\subset E\subset E'$. 
    Put $\mathbf{b}' = \set{b_1',\ldots ,b_s'} \in K_s^M(E'(Q'))$. 
    Since the isomorphism \eqref{eq:MK-Chow} is canonical, 
    there is a commutative diagram:
    \[
    \xymatrix{
    \displaystyle{\bigoplus_{y'\in (X_{E'})_0} }K_s^M(E'(y')) \ar[r]\ar[d]_-{N_{E'(y')/E(y)}} & CH^{d+s}(X_{E'},s)\ar[d]^{j_{\ast}}  \\
    \displaystyle{\bigoplus_{y\in (X_{E})_0} }K_s^M(E(y)) \ar[r] & CH^{d+s}(X_{E},s). 
    }
    \]
    This gives 
    \begin{equation}
        \label{eq:jastQ'}
        j_{\ast}([Q']) = N_{E'(Q')/E(Q)}(\bb')  \quad \mbox{in $CH^{d+s}(X_E,s)$}.
    \end{equation}
    Then we have 
    \begin{align*}
		\psi_{E'}(\ba\otimes [Q']) 
        &= \Cor_{E'(Q')/F}(\AS{\ba, b'_1, \ldots, b_s'}_{E'(Q')}) \\
		&= \Cor_{E(Q)/F}\circ \Cor_{E'(Q')/E(Q)} \circ \tilde{s}_{E'(Q'),p^r}^s\left(\set{\ba, \bb'}_{E'(Q')/E'(Q')}\right) \\
        &\stackrel{(\mathbf{PF})}{=}\Cor_{E(Q)/F} \circ \tilde{s}_{E(Q),p^r}^s \left(\set{\ba, N_{E'(Q')/E(Q)}(\bb')}_{E(Q)/E(Q)}  \right) \\
		&\stackrel{\eqref{eq:jastQ'}}{=}  
        \psi_{E}(\ba\otimes j_{\ast}[Q']).
    \end{align*}
    {\red \linelabel{com:9}Recall that the target $KH_0^{(s)}(X,\Z/p^r)$ of the map $\psi$ (cf.~\eqref{eq:psiE}) is a quotient of $\bigoplus_{x\in X_0}H^{s+1}_{p^r}(F(x))$.  
    The latter group is isomorphic to $\bigoplus_{x\in X_0}(W_r\otimesM K_s^M)(F(x))/\wp$ 
    by \autoref{thm:H24} and \cite[Thm.~4.5]{Hir24} as noted in the beginning of this proof.}
    From the construction of the map $\psi$, 
    the induced homomorphism $\psi\colon (W_r\otimesM \underline{CH}^{d+s}(X,s))(F) \to KH_0^{(s)}(X,\Z/p^r)$ is surjective. 
    Finally, the extended differential symbol kills the image of $\wp$ and hence 
    $\psi$ above induces the required homomorphism. 
\end{proof}

The map $f_{\ast}^{CH} \colon \underline{CH}^{d+s}(X,s) \to K_s^M$ 
induces 
\begin{equation}
    \label{eq:fastM}
    f_{\ast}^M := \Id\otimes {f_{\ast}^{CH}}\colon \Big(W_r\otimesM \underline{CH}^{d+s}(X,s)\Big)(F)/\wp \to \Big(W_r\otimesM K_s^M\Big)(F)/\wp.
\end{equation}
There is a commutative diagram: 
\begin{equation}
    \label{diag:fastM-KH}
\vcenter{
\xymatrix@C=15mm{
\Bigl(W_r\otimesM \underline{CH}^{d+s}(X,s)\Bigr)(F)/\wp \ar@{->>}[d]_{\psi}\ar[r]^-{f_{\ast}^M} & \Big(W_r\otimesM K_s^M\Big)(F)/\wp \ar[d]^{\simeq} \\
KH_0^{(s)}(X,\Z/p^r)\ar[r]_-{f_{\ast}^{KH}}& H_{p^r}^{s+1}(F).
}}
\end{equation}
Here, the right vertical map is bijective because of 
$\Gm^{\otimesM s}/p^r \simeq  K_s^M/p^s$ (cf.~\cite[Cor.~5.11]{Hir24}).

\begin{lem}\label{lem:varphi}
The homomorphisms 
$f_{\ast}^M$ is surjective. 
In particular, $f_{\ast}^{KH}$ is surjective.
\end{lem}
\begin{proof}
First, we show that $f_{\ast}^M$ is surjective. 
By \autoref{lem:normGa} (2), we may assume $r=1$. 
By \cite[Cor.~5.11]{Hir24}, 
$(\Ga\otimesM K_s^M)(F)/\wp$ is generated by elements of the form $\set{a,x}_{F/F}$. 
Take a symbol $\set{a,x}_{F/F}$. 
There exists a finite separable field extension $F'/F$ such that 
$X(F') \neq \emptyset$ (\cite[Prop.~3.2.20]{Liu02}). 
For the base change $X_{F'} = X\otimes_F F'$, 
the map $f_{\ast}^{CH}\colon CH^{d+s}(X_{F'},s) \to K_s^M(F')$ is surjective 
because of $X_{F'}(F') \simeq X(F') \neq \emptyset$. 
Therefore, there exists $\xi'\in CH^{d+s}(X_{F'},n)$ such that 
$\res_{F'/F}(x) = f_{\ast}^{CH}(\xi')$. 
As the extension $F'/F$ is separable, there exists $a'\in F'$ such that $\Tr_{F'/F}(a') = a$. 
By the projection formula (\textbf{PF}), we have 
\[
\set{a,x}_{F/F} = \set{\Tr_{F'/F}(a'), x}_{F/F}
    = \set{a', \res_{F'/F}(x)}_{F'/F}
    = \set{a', f_{\ast}^{CH}(\xi')}_{F'/F}.
\]
This implies that $\set{a', \xi'}_{F'/F} \in (\Ga\otimes \underline{CH}^{d+s}(X,n))(F)/\wp$ 
maps to $\set{a,x}_{F/F}$ and hence 
$f_{\ast}^M$ is surjective. 

The last assertion follows from the commutative diagram \eqref{diag:fastM-KH}.
\end{proof}

\begin{prop}
\label{lem:isom}
Let $F$ be a field of  $\ch(F) = p>0$ 
with $[F:F^p]\le p^s$ for some $s$. 
Let $X$ be a \emph{projective smooth and geometrically irreducible scheme} over $F$ of $d = \dim(X)$.
    {\red For a non-negative integer $t$,}
    we consider the following conditions: 
    \begin{itemize}
    \item[{\red $(\mathrm{Van}_t)$}] $\Big(\Ga\otimesM \underline{A}^{d+t}(X,t)\Big)(F)/\wp = 0$.  
    \item[$(\mathrm{Pt})$] There exists a finite separable field extension $F'/F$ such that $X(F')\neq \emptyset$ and the extension degree $[F':F]$ is prime to $p$.
    \item[{\red $(\mathrm{Cor}_t)$}] For any finite separable field extension $F'/F$, 
    the corestriction $\Cor_{F'/F}\colon H^{t+1}_{p}(F') \to H^{t+1}_{p}(F)$ is injective. 
    \end{itemize}
    If we assume the condition $(\mathrm{Pt})$ or $(\mathrm{Cor}_s)$ {\red for the integer $s$}, 
    then there is a right exact sequence \linelabel{com:10}
         \begin{equation}
	   \label{seq:A}
        \frac{\Big(W_r\otimesM \underline{A}^{d+s}(X,s)\Big)(F)}{\wp} \to \frac{\Big(W_r\otimesM \underline{CH}^{d+s}(X,s)\Big)(F)}{\wp} \xrightarrow{f_{\ast}^M } \frac{\Big(W_r\otimesM K_s^M\Big)(F)}{\wp}{\red \to 0}.
    \end{equation}
    Furthermore, if we additionally assume the condition $(\mathrm{Van}_s)$, then the map 
    \[
        \psi\colon \Big(W_r\otimesM\underline{CH}^{d+s}(X,s)\Big)(F)/\wp \xrightarrow{\simeq} KH_0^{(s)}(X,\Z/p^r)
    \]
    and the natural map 
    \[
    f_{\ast}^{KH}\colon KH_0^{(s)}(X,\Z/p^r) \xrightarrow{\simeq} KH_0^{(s)}(\Spec(F),\Z/p^r) = H_{p^r}^{s+1}(F)
    \]
    are isomorphisms. The complex \eqref{eq:RecCpx} is exact. 
\end{prop}
\begin{proof}
First, we assume the condition (Pt), that is, 
there exists a finite separable extension $F'/F$ such that 
$X(F')\neq\emptyset$ and $[F':F] = m$ is prime to $p$. 
By $X(F')\neq \emptyset$, 
the map $f_{\ast}^{CH}\colon CH^{d+s}(X_{F'},s) \to K_s^M(F')$ is surjective over $F'$. 
Consider the following commutative diagram:
\[
\xymatrix{
CH^{d+s}(X_{F'},s)/p^r \ar[d]_{\tr_{F'/F}}\ar@{->>}[r]^-{f^{CH}_{\ast}} & K_s^M(F')/p^r\ar[d]^{N_{F'/F}} \\
CH^{d+s}(X,s)/p^r \ar[r]^-{f^{CH}_{\ast}} & K_s^M(F)/p^r.
}
\]
Since $[F':F]= m$ is prime to $p$, the compositions 
$\tr_{F'/F}\circ \res_{F'/F} = m$ on $CH^{d+s}(X,s)/p^r$ 
and $N_{F'/F}\circ \res_{F'/F}  = m$ on $K_s^M(F)/p^r$ are bijective. 
The norm map $N_{F'/F}\colon K_s^M(F')/p^r \to K_s^M(F)/p^r$ is surjective, 
so is $f^{CH}_{\ast}\colon CH^{d+s}(X,s)/p^r\to K_s^M(F)/p^r$. 
One can extend this map to a surjective map $f_{\ast}^{CH}\colon \underline{CH}^{d+s}(X,s)/p^r\to K_s^M/p^r$. 
In particular, 
\[
\underline{A}^{d+s}(X,s)/p^r \to \underline{CH}^{d+s}(X,s)/p^r \xrightarrow{f^{CH}_{\ast}} K_s^M/p^r \to 0 
\]
is exact.
From the isomorphisms  
\begin{gather*}
    \Big(W_r/\wp\otimesM K_s^M/p^r\Big) (F) \simeq \Big(W_r\otimesM K_s^M\Big)(F)/\wp,\quad \mbox{and}\\
    \Big(W_r/\wp\otimesM \underline{CH}^{d+s}(X,s)/p^r\Big) (F) \simeq \Big(W_r\otimesM \underline{CH}^{d+s}(X,s)\Big)(F)/\wp    
\end{gather*}
(cf.~\cite[Lem.~5.1]{Hir24}), 
the sequence \eqref{seq:A} is right exact.

Next, we assume the condition ($\mathrm{Cor}_s$) for $t =s$. 
Take a finite separable field extension $F'/F$ such that 
$X(F') \neq \emptyset$  (\cite[Prop.~3.2.20]{Liu02}).  
From the argument just above, 
the sequence \eqref{seq:A} is right exact over $F'$. 
There is a commutative diagram with exact rows:\linelabel{com:11}
\[
\xymatrix{
0 \ar[r] & H^{s+1}_{p^{r-1}}(F') \ar[r]\ar[d]^{\Cor_{F'/F}} & H^{s+1}_{p^{r}}(F')\ar[r]\ar[d]^{\Cor_{F'/F}} & H^{s+1}_p(F')\ar[d]^{\Cor_{F'/F}} \\
0 \ar[r] & H^{s+1}_{p^{r-1}}(F) \ar[r] & H^{s+1}_{p^{r}}(F)\ar[r] & H^{s+1}_p(F) 
}
\]
(for the exactness of the horizontal sequences, see \cite[Proof of Thm.~5.8]{Hir24}). 
By the assumption, and the induction on $r$, 
the corestriction $\Cor_{F'/F}\colon H^{s+1}_{p^r}(F') \to H^{s+1}_{p^r}(F)$ 
is injective.
By \autoref{lem:varphi}, 
$f_{\ast}^M$ is surjective so that 
the natural map 
$\Big(W_r\otimesM \operatorname{Im}(f_{\ast}^{CH})\Big)(F)/\wp \twoheadrightarrow \Big(W_r\otimesM K_s^M\Big)(F)/\wp$ is surjective.
Consider the following commutative diagram: \linelabel{com:12}
\[
\xymatrix{
\Big(W_r\otimesM\operatorname{Im}(f_{\ast}^{CH})\Big)(F')/\wp \ar@{->>}[d]_{\tr_{F'/F}}\ar[r]^-{\simeq} &\Big(W_r\otimesM K_s^M\Big)(F')/\wp\ar@{->>}[d]^{\tr_{F'/F}} \ar[r]^-{\simeq} & H_{p^r}^{s+1}(F')\ar@{^{(}->}[d]^{\Cor_{F'/F}}\\
\Big(W_r\otimesM\operatorname{Im}(f_{\ast}^{CH})\Big)(F)/\wp \ar@{->>}[r] & \Big(W_r\otimesM K_s^M\Big)(F)/\wp\ar[r]^-{\simeq} & H^{s+1}_{p^r}(F).
}
\]
Here, the left and the middle vertical maps are surjective by \autoref{lem:normGa} (1). 
From the above diagram, the sequence \eqref{seq:A} is right exact. 

Finally, 
the assumption ($\mathrm{Van}_s$) and \autoref{lem:normGa} (3) imply $(W_r\otimesM \underline{A}^{d+s}(X,s))(F)/\wp = 0$ 
for any $r\ge 0$.
From the exact sequence \eqref{seq:A}, the natural map $f_{\ast}^M$ is bijective.  
By the commutative diagram \eqref{diag:fastM-KH}, 
both of $\psi$ and $f_{\ast}^{KH}$ are bijective. 
\end{proof}

\subsection*{Finite fields}
We consider the case where the field $F$ is perfect, that is, $[F:F^p] =1$. 
\begin{prop}
Suppose that $F$ is perfect. 
\begin{enumerate}
    \item 
For $s>0$, we have 
\[
\Big(W_r\otimesM \underline{CH}_0(X)\otimesM \Gm^{\otimesM s} \Big)(F)/\wp=KH_0^{(s)}(X,\Z/p^r) = 0.
\]
\item If $F$ is algebraically closed, then the equalities in $\mathrm{(1)}$ holds for $s\ge 0$.
\end{enumerate}
\end{prop}
\begin{proof}
(1) 
By \autoref{thm:Akh2}, there is a natural surjective homomorphism 
\begin{equation}\label{eq:quot}	
 \Big(W_r\otimesM\underline{CH}_0(X)\otimesM \Gm^{\otimesM s}\Big)(F)\twoheadrightarrow 
 \Big(W_r\otimesM \underline{CH}^{d+s}(X,s)\Big)(F).
\end{equation}

The assertion follows from 
$(W_r\otimesM \Gm)(F) = 0$ (\cite[Lem.~2.2]{Hir14}) and \autoref{prop:surj}.

\noindent
(2) Take any symbol $\set{a,x}_{F/F} \in \Big(W_r\otimesM\underline{CH}_0(X)\Big) (F)$. 
There exists $\alpha \in F$ such that $\wp (\alpha)=a$. 
As a result, the symbol $\set{a,x}_{F/F} =0$ in  $\Big(W_r\otimesM\underline{CH}_0(X)\Big) (F)/\wp$. 
\end{proof}


\begin{conj}[{\cite[Conj.~0.3]{Kat86}}]
\label{conj:Kato}
If $F$ is a finite field, 
then  we have 
	\begin{equation}
	\label{eq:KH}
		KH_j^{(0)}(X,\Z/p^r) \xrightarrow{\simeq} \begin{cases}
			\Z/p^r, &j=0, \\
			0, &j>0.
		\end{cases} 
	\end{equation}
\end{conj}
The above conjecture is known for the case $\dim(X)\le 2$ (\cite{Kat86}). 
For a higher dimensional scheme $X$, 
Jannsen and Saito (\cite[Thm.~0.3]{JS09}) proved the above conjecture 
under some assumptions on the resolutions of singularities. 
In particular, the isomorphism \eqref{eq:KH} holds for $j\le 4$ unconditionally.

\begin{thm}\label{thm:finite}\linelabel{com:13}
    Suppose that $F$ is finite. 
    The conditions $(\mathrm{Van}_0)$ and $(\mathrm{Cor}_0)$ for $t=0$ in \autoref{lem:isom} hold. 
    In particular, 
    we have isomorphisms 
    \begin{gather*}
        \psi\colon \Big(W_r\otimesM \underline{CH}_0(X)\Big)(F)/\wp \xrightarrow{\simeq} KH_0^{(0)}(X,\Z/p^r),\ \mbox{and} \\
     f_{\ast}^{KH}\colon KH_0^{(0)}(X,\Z/p^r) \xrightarrow{\simeq} H_{p^r}^{1}(F) \simeq \Z/p^r.
     \end{gather*}
\end{thm}
\begin{proof}
 \linelabel{com:14}
    For a finite extension $F'/F$, 
    the corestriction $\Cor_{F'/F}\colon H^1_{p}(F') \to H^1_{p}(F)$ is 
    bijective, in fact, 
    this corresponds to the trace $\Tr_{F'/F}\colon F'/\wp \to F/\wp$ which is surjective 
    and $F'/\wp \simeq F/\wp \simeq \Z/p$. 
    \linelabel{com:15} {\red Thus,} the condition ($\mathrm{Cor}_0$) in \autoref{lem:isom} holds. 
	We show the condition ($\mathrm{Van}_0$).
	By \autoref{lem:normGa} (1) we can replace $F$ with a finite field extension $F'/F$. 
	Thus, we may assume $X(F)\neq \emptyset$. 
   	By the higher dimensional class field theory 
    (\cite[Sect.~10, Prop.~9]{KS83b}), there is a short exact sequence 
    \[
        0 \to T(X) \to A_0(X) \to \mathrm{Alb}_{X}(F)\to 0, 
    \]
    where $A_0(X) = \underline{A}^d(X,0)(F)$, 
    $\mathrm{Alb}_{X}$ is the Albanese variety of $X$, and $T(X)$ is defined by the exactness. 
    The kernel $T(X)$ is a finite group satisfying $\varinjlim_{E/F} T(X_E) = 0$, 
    where $E/F$ runs through the set of all finite extensions of $F$ in $\overline{F}$ 
    and $X_E := X\otimes_F E$. 
    The above short exact sequence gives rise to the short exact sequence of Mackey functors: 
    \[
    0 \to \underline{T}(X) \to \underline{A}_0(X)\to \mathrm{Alb}_{X} \to 0.
    \]
    By applying $\Ga\otimesM -$, 
    we have 
    \[
        \Big(\Ga\otimesM \underline{T}(X)\Big)(F)\to \Big(\Ga\otimesM \underline{A}_0(X)\Big)(F)\to (\Ga\otimesM \mathrm{Alb}_{X})(F)\to 0. 
    \]
    We have $\Big(\Ga\otimesM \mathrm{Alb}_X\Big)(F) = 0$ (\cite[Lem.~2.2]{Hir18}). 
    It is enough to show $(\Ga\otimesM \underline{T}(X))(F)/\wp = 0$. 
    By the norm arguments, it is sufficient to show 
    a symbol  of the form $\set{a, b}_{F/F}$ in $(\Ga \otimesM \underline{T}(X))(F)/\wp $ 
    is trivial. 
    There exists a finite extension $E/F$ such that 
    $\res_{E/F}(b) = 0$ in $T(X_E)$. 
    The trace map $\Tr_{E/F}\colon E\to F$ is surjective so that one can find 
    $\alpha \in E$ with  
    $\Tr_{E/F}(\alpha) = a$. 
    We obtain
    $\set{a,b}_{F/F} = \set{\Tr_{E/F}(\alpha),b}_{F/F} = \set{\alpha,\res_{E/F}(b)}_{E/F} = 0$. 
\end{proof}

\subsection*{Local fields}
In the following, we assume $[F:F^p] = p$ 
such as a local field or a function field of one variable over a finite field.

\begin{prop}
Suppose $[F:F^p]=p$. 
For $s\ge 2$, we have 
\[
\Big(W_r\otimesM  \underline{CH}_0(X)\otimesM \Gm^{\otimesM s} \Big)(F)/\wp = KH_0^{(s)}(X,\Z/p^r) = 0.
\]
\end{prop}
\begin{proof}
Since $K_2^M(F)$ is $p$-divisible (\cite[Prop.~5.13]{BT73}), 
it follows that 
$(\Gm \otimesM \Gm)(F)/p^r \simeq K_2^M(F)/p^r$ (\cite[Thm.~4.5]{Hir24}) is trivial. 
By \autoref{prop:surj}, 
we then conclude that 
\[
\Big(W_r\otimesM  \underline{CH}_0(X)\otimesM \Gm^{\otimesM s} \Big)(F)/\wp = KH_0^{(s)}(X,\Z/p^r) = 0.
\]
\end{proof}

Recalling from \eqref{eq:MK-Chow}, we have 
\begin{equation}
	\label{def:SK1}
CH^{d+1}(X,1) \simeq \Coker\left(\partial \colon \bigoplus_{x\in X_1} K_2^M(F(x)) \to \bigoplus_{x\in X_0} F(x)^{\times}\right).
\end{equation}
{\red The right hand side is often denoted by $SK_1(X)$ (\cite[Sect.~1]{Blo81}, \cite[Sect.~1]{JS03}).} 
The structure map $f\colon X\to \Spec(F)$ induces a left exact sequence  
\[
  0\to V(X) \to CH^{d+1}(X,1) \xrightarrow{f_{\ast}^{CH}} F^{\times}, 
\]
where $V(X) := A^{d+1}(X,1)$ and the map $f_{\ast}^{CH}$ is defined by the norm maps 
$N_{F(x)/F}\colon F(x)^{\times} \to F^{\times}$  for closed points $x$ in $X$.
We also denote by 
$\underline{V}(X)$ the kernel of the  map $f_{\ast}^{CH}\colon \underline{CH}^{d+1}(X,1) \to \Gm$ 
of Mackey functors.

\begin{lem}
\label{lem:GaVXprod}
    Suppose that 
    $X = X_1\times  \cdots \times X_d$ 
    is the product of projective smooth and geometrically irreducible curves $X_1,\ldots, X_d$ over $F$. 
    If we assume $\Big(\Ga\otimes \underline{V}(X_i)\Big)(F)/\wp = 0$ for each $1\le i\le d$, 
    then we have $\Big(\Ga\otimes \underline{V}(X)\Big)(F)/\wp = 0$. 
\end{lem}
\begin{proof}
    By \autoref{lem:normGa} (1), 
    we may assume that $X_i(F)\neq \emptyset$ for all $i$. 
    By \cite[Thm.~2.2]{Yam09}, there is an isomorphism 
	\[
	\underline{V}(X) \simeq \bigoplus_{r=1}^d \bigoplus_{1\le i_1< \cdots < i_r\le d}K(-;J_{i_1},\ldots,J_{i_r},\Gm)
	\]
	as Mackey functors, 
	where $J_i := \Jac_{X_{i}}$ is the Jacobian variety of $X_i$. 
	By renumbering if necessary, 
	it is enough to show $\Big(\Ga\otimesM K(-;J_1,\ldots, J_j,\Gm)\Big)(F)/\wp = 0$ for $j>1$. 
	By \autoref{lem:product}, there is a surjection 
	\[
	 K(-;J_1,\Gm) \otimesM J_2\otimesM \cdots\otimesM J_j \twoheadrightarrow K(-;J_1,J_2,\ldots,J_j,\Gm).
	\]
    We have $0 = \Big(\Ga \otimesM \underline{V}(X_i)\Big)(F)/\wp = \Big(\Ga\otimesM K(-;J_1,\Gm)\Big)(F)/\wp$, 
    the assertion follows from this.
\end{proof}

Now, we suppose that the base field $F$ is a local field (of $\ch(F) = p>0$) with finite residue field $k$.  
We assume that there exists a model $\sX$ of $X$ defined over the valuation ring $\cO_F$ of $F$ 
and {\red its special fiber is denoted by $X_k$.} \linelabel{com:17}

\begin{conj}[{\cite[Conj.~5.1]{Kat86}, \cite[Conj.~B]{JS03}}] \label{conj:Kato2}
    Assume that the model $\sX$ is proper over $\cO_F$ and regular. 
    Then, the residue map    
    \begin{equation}\label{eq:res}
        \Delta_j: KH_j^{(1)}(X,\Z/p^r)\to KH_j^{(0)}(X_k,\Z/p^r)
    \end{equation}
    is an isomorphism for all $r\ge 1$ and $j\ge 0$.
\end{conj}

The above conjecture is known in the case $\dim(X) = 1$ (\cite{Kat86}).

\begin{lem}
\label{lem:GaVXlocal}
    Suppose that $F$ is a local field with finite residue field $k$, 
    and $\dim (X) = 1$. Then, we have 
    \[
    \Big(\Ga \otimesM \underline{V}(X)\Big)(F)/\wp = 0.
    \]
\end{lem}
\begin{proof}    
    To show $\Big(\Ga \otimesM \underline{V}(X)\Big)(F)/\wp = 0$, 
    we can replace $F$ by a finite separable field extension by \autoref{lem:normGa} (1). 
    As we have $X(F^{\sep})\neq \emptyset$, 
    by replacing $F$, we may assume that $X(F)\neq\emptyset$. 
    Let $J = \Jac_X$ be the Jacobian variety of $X$. 
By the class field theory for $X$ (\cite[Thm.~1.1, Thm.~5.1]{Yos03}, \cite[II, Cor.~4.4]{Sai85a}), 
the reciprocity map induces a short exact sequence
\begin{equation}
    \label{eq:geofg}
0 \to V(X)/V(X)_{\mathrm{div}}  \to \pi_1^{\ab}(X)^{\geo} \to \Zhat^{r(X)} \to 0,
\end{equation}
where $V(X)_{\mathrm{div}}$ is the maximal divisible subgroup of $V(X)$, 
and $\pi_1^{\ab}(X)^{\geo}$ is the kernel of the map $\pi_1^{\ab}(X)\to \pi_1^{\ab}(\Spec(F))= G_F^{\ab}$ 
induced from the structure map $X\to \Spec(F)$. Here, 
the quotient $V(X)_{\mathrm{fin}} := V(X)/V(X)_{\mathrm{div}}$ is a finite group,  
and $r(X)\ge 0$ is the \emph{rank} of $X$ (\cite[II, Def.~2.5]{Sai85a}).\linelabel{com:19} 

{\red By \cite[II, Sect.~6]{Sai85a},\linelabel{com:20} 
for any finite extension $F'/F$, the rank of $X_F' = X\otimes_FF'$ satisfies the inequalities 
\[
r(X) \le r(X_{F'}) \le 2g,
\]
where $g$ denotes the genus of $X$. 
According to \cite[II, Lem.~4.5]{Sai85a}, 
there exists a finite extension $F'/F$ with residue field $k'$ 
such that the base change $X_{F'}$ admits a regular model $\sX$ over $\cO_{F'}$ 
whose special fiber $\sX_s$ satisfies the following conditions: 
(a) $\sX_s$ is reduced; (b) $\sX_s$ has only $k'$-rational ordinary double points as singularities; 
and (c) each irreducible component of $\sX_s$ is regular and geometrically irreducible over $k'$. 
In fact, the proof of \cite[II, Lem.~4.5]{Sai85a} 
uses the semi-stable reduction theorem for curves so that 
the extension $F'/F$ above can be taken to be separable (cf.~\cite[Thm.~10.4.3]{Liu02}, \cite[\href{https://stacks.math.columbia.edu/tag/0CDN}{Tag 0CDN}]{stacks-project}). 
For our purposes, we may thus replace $F$ by such an extension $F'$. 

In the proof of \cite[II, Thm.~6.2]{Sai85a}, 
the special fiber $\sJ_s^0$ of the identity component of the N\'eron model $\sJ$ of the Jacobian variety $J$ 
fits into an exact sequence 
\[
0 \to \Gm^{r(X)}\to \sJ_s^0\to \prod_{i=1}^n (J_i)_k \to 0, 
\]
where each $(J_i)_k$ denotes the Jacobian variety of an irreducible component $C_i$ of $\sX_s$.
This yields the equality $r(X) = g - \sum_{i=1}^ng_i$, where $g_i$ denotes the genus of $C_i$. 
It follows that $r(X)$ is the maximal value among all 
$r(X_{F'})$ as $F'$ ranges over finite extensions of $F$ 
(a value referred to as the \textbf{geometric rank} of $X$ in \cite[II, Def.~6.1]{Sai85a}). 
In fact, for any finite extension $F'/F$, 
the special fiber $\sX_s'$ of the base change $\sX' := \sX_{\cO_{F'}}$ of $\sX$ to $\Spec(\cO_{F'})$ satisfies the conditions 
(a), (b), and (c) above. 
The irreducible components of $\sX_s'$ are the base changes 
$C_i\otimes_kk'$ of the components $C_i$.}

Take any symbol of the form $\set{a,x}_{F/F}$ in $\Big(\Ga \otimesM \underline{V}(X)\Big)(F)/\wp$. 
By the norm arguments, the equality $\set{a,x}_{F/F} = 0$ implies the assertion. 
Consider the Artin-{\red Schreier}\linelabel{com:18} extension $E := F(\wp^{-1}(a))$ of $F$ 
which is a cyclic $p$-extension. 

{\red As noted above, the rank $r(X_E)$ of $X_E$ is same as that of $X$. 
Putting $r := r(X) = r(X_E)$,} the following diagram is commutative
\[
\xymatrix{
0 \ar[r] & V(X_{E})_{\mathrm{fin}} \ar[d]_{N_{E/F}} \ar[r] & (T^{\et}(J))_{G_{E}}\ar@{->>}[d]\ar[r] & \Zhat^{r} \ar@{=}[d] \ar[r] & 0 \\
0 \ar[r] & V(X)_{\mathrm{fin}} \ar[r] & (T^{\et}(J))_{G_F} \ar[r] & \Zhat^{r} \ar[r] & 0,
}
\]
where the left vertical map is the norm map $N_{E/F} \colon V(X_E) \to V(X)$, 
the middle map is the natural map arising from $G_E \hookrightarrow G_F$ 
and is surjective. 
Accordingly, the induced map $N_{E/F}\colon V(X_E)/p \twoheadrightarrow V(X)/p$ 
is surjective. 
Since we have $N_{E/F}(\xi) = x$ for some $\xi \in V(X_E)/p$, 
the projection formula gives 
\[
\set{a,x}_{F/F} = \set{a,N_{E/F}(\xi)}_{F/F} = \set{\res_{E/F}(a),\xi}_{E/F} = 0.
\]
\end{proof}

We recall that a curve 
$C$ over $F$ is said to have \textbf{good reduction}\linelabel{com:21}
if the special fiber $C_{k} = \sC \otimes_F k$ of a model $\sC$ of $C$ is smooth over the finite residue field $k$. 

\begin{thm}\label{thm:local}
    Suppose that $F$ is a local field with finite residue field $k$, 
    and $X  = X_1\times \cdots \times X_d$ 
    is the product of projective smooth and geometrically irreducible curves $X_i$ over $F$. 
    \begin{enumerate}
    \item 
    The conditions $(\mathrm{Van}_1)$ and $(\mathrm{Cor}_1)$ for $t=1$ in \autoref{lem:isom} hold. 
    In particular, 
    for any $r\ge 1$, we have isomorphisms 
    \begin{gather}
            \Big(W_r\otimesM \underline{CH}^{d+1}(X,1)\Big)(F)/\wp \simeq KH_0^{(1)}(X,\Z/p^r),\quad \mbox{and}\\
            f_{\ast}^{KH}\colon KH_0^{(1)}(X,\Z/p^r) \xrightarrow{\simeq} H_{p^r}^{2}(F).
    \end{gather}
    \item 
    We further assume that each $X_i$ has good reduction. 
    Then, the residue map 
    \[
        \Delta_0\colon KH_0^{(1)}(X,\bZ/p^r) \xrightarrow{\simeq} KH_0^{(0)}(X_k,\bZ/p^r)
    \] 
    is bijective. 
    \end{enumerate}	
\end{thm}
\begin{proof}
(1) 
 	Note that 
 $H^2_{p^r}(F) \simeq  \Br(F)[p^r]$ is the $p^r$-torsion part of the Brauer group of $F$ {\red (cf.~\eqref{eq:HBr} in \autoref{sec:intro})\linelabel{com:22}}. 
 The condition ($\mathrm{Cor}_1$) in \autoref{lem:isom} follows from the fact 
 	that the  
 corestriction of the Brauer groups $\Cor_{F'/F} \colon \Br(F')[p] \to \Br(F)[p]$ is bijective 
  for a finite separable field extension $F'/F$. 
    We have
	$\underline{A}^{d+1}(X,1) = \underline{V}(X)$. 
    By \autoref{lem:GaVXlocal} and \autoref{lem:GaVXprod}, 
    we have $\Big(\Ga\otimesM \underline{A}^{d+1}(X,1)\Big)(F)/\wp  = 0$. 
	
\smallskip
\noindent
(2)
There is a commutative diagram
\[
\xymatrix{
KH_0^{(1)}(X,\Z/p^r)\ar[d]_{\Delta_0}\ar[r]^-{f_{\ast}^{KH}}_-{\simeq }&
H^2_{p^r}(F)\ar[d]^{\partial}\\
KH_0^{(0)}(X_k,\Z/p^r)\ar[r]^-{\simeq}
&
H_{p^r}^1(k), 
}
\]
where the right vertical map $\partial\colon H^2_{p^r}(F) = \Br(F)[p^r] \to H^1_{p^r}(k) = H^1(k,\Z/p^r)$ is bijective. 
This diagram shows that $\Delta_0$ is an isomorphism.
\end{proof}

\subsection*{Global fields}
Next, we suppose that the field $F$ is a \emph{global field}, 
that is, a function field of one variable 
over a finite field $k$ of  $\ch(k) = p>0$. 
For any place $v$ of $F$, 
we denote by $F_v$
the completion of $F$ with respect to the place $v$. 

\begin{conj}[{\cite[Conj.~0.4]{Kat86}, \cite[Conj.~A]{JS03}}]
\label{conj:Kato3}
Suppose that $F$ is a global field. 
Then, for $j>0$, there is an isomorphism 
\[
KH_j^{(1)}(X,\Z/p^r) \xrightarrow{\simeq} \bigoplus_v KH_j^{(1)}(X_{v},\Z/p^r) 
\]	
and, for $j=0$, the short exact sequence below exists:
\[
0\to KH_0^{(1)}(X,\Z/p^r) \to \bigoplus_v KH_0^{(1)}(X_v,\Z/p^r) \to \Z/p^r \to 0, 
\]
where $v$ runs over the set of the places of $F$ and $X_v := X\otimes_F F_v$ for each place $v$. 
\end{conj}

The conjecture above holds when $\dim(X) = 1$ (\cite{Kat86}).

\begin{lem}
\label{lem:GaVX}
    Suppose that $F$ is a global field 
    and $\dim (X) = 1$.
    Then, we have 
    \[
    \Big(\Ga \otimesM \underline{V}(X)\Big)(F)/\wp = 0.
    \]
\end{lem}
\begin{proof}
    To show $\Big(\Ga \otimesM \underline{V}(X)\Big)(F)/\wp = 0$, 
    we can replace $F$ with a finite separable field extension (\autoref{lem:normGa} (1)). 
    From this, we may assume that 
    $X(F) \neq \emptyset$, 
    and, 
    there exists a finite set $S$ of places in $F$ such that 
    $X$ has good reduction outside $S$, and 
    for each $v\in S$, 
    the special fiber of the connected component of the N\'eron model of 
    the Jacobian variety $\Jac_{X_v}$ of $X_v := X\otimes_F F_v$ is an extension of an abelian variety by a split torus.
    
    As in the proof of \autoref{lem:GaVXlocal}, 
    take a symbol $\set{a,x}_{F/F}$ in $\Big(\Ga \otimesM \underline{V}(X)\Big)(F)/\wp$ 
    and put $E := F(\wp^{-1}(a))$. 
    To show $\set{a,x}_{F/F} = 0$, it is sufficient to show that 
    the norm $V(X_E)/p\to V(X)/p$ is surjective, 
    where $X_E := X\otimes_F E$. 
    There is a spectral sequence 
    \[
    E_2^{i,j} = H^i(F, H^j(X_{F^{\sep}}, \Omega^2_{\log})) \Rightarrow H^{i+j}(X,\Omega^2_{\log})
    \]
    (cf.\  \cite[Proof of Prop.~4]{KS83b}). 
    We have $E_2^{i,j} = 0$ for $i\ge 2$ by the cohomological reason (\cite[Prop.~6.5.10]{NSW08}) 
    and for $j\ge 2$ by \cite[Lem.~1, (2)]{KS83b}.
    The edge homomorphisms give the bottom horizontal exact sequence in the commutative diagram below:
    \begin{equation}
        \label{diag:edge}
        \vcenter{
    \xymatrix{
    0\ar[r] & V(X)/p \ar[r]\ar[d]^{\tau} & CH^{2}(X,1)/p \ar[d]^{\rho}\ar[r]^-{f_{\ast}^{CH}} & F^{\times}/p\ar[d]\ar[r] & 0 \\
    0\ar[r]  & H^1(F,M) \ar[r] & H^1(X,\Omega_{\log}^2) \ar[r]  & H^0(F,H^1(X_{F^{\sep}},\Omega_{\log}^2)) \ar[r] & 0,
    }}
    \end{equation}
    where $M = H^0(X_{F^{\sep}}, \Omega^2_{\log})$, 
    and the middle vertical map $\rho$ is given in \cite[Lem.~2.4]{Kat86} which is injective and its cokernel is 
    \[
    \Coker(\rho) \simeq \Ker\left(\partial\colon H^3(F(X), \Z/p(2))\to \bigoplus_{x\in X_0} H^2(F(x), \Z/p(2))\right)
    = KH_1^{(1)}(X,\Z/p). 
    \]
    Here, 
    the right vertical map is bijective due to $H^1(X_{F^{\sep}},\Omega_{\log}^2) \simeq (F^{\sep})^{\times}/p$ (\cite[Lem.~1 (2)]{KS83b}). 
    Since $X(F)\neq \emptyset$ and 
    $F^{\times}$ is $p$-torsion free, the top sequence in \eqref{diag:edge} is split exact.
    Therefore, we obtain a short exact sequence 
    \[
    0 \to V(X)/p \to H^1(F,M) \to KH_1^{(1)}(X,\Z/p) \to 0.
    \]
    Consider the commutative diagram with exact rows: 
    \begin{equation}
        \label{diag:Cor}
        \vcenter{
    \xymatrix{
    0 \ar[r] & V(X_E)/p \ar[d]^{N_{E/F}} \ar[r] & H^1(E,M) \ar@{->>}[d]^{\Cor_{E/F}} \ar[r] & KH_1^{(1)}(X_E,\Z/p)\ar[d] \ar[r] & 0 \\
    0 \ar[r] & V(X)/p \ar[r] & H^1(F,M) \ar[r] & KH_1^{(1)}(X,\Z/p) \ar[r] & 0.
    }}
    \end{equation}
    Here, the middle vertical map is surjective. 
    In fact, 
    the natural homomorphism $\phi\colon \mathrm{Ind}_F^E(M) \to M$ of $G_F$-modules 
    induces 
    $H^1(F,\mathrm{Ind}_F^E(M)) \to H^1(F,M) \to H^2(F,\Ker(\phi)) = 0$ 
    and this compatible with the corestriction $\Cor_{E/F}\colon H^1(E,M) \to H^1(F,M)$ 
    through 
    the isomorphism $H^{\ast}(F,\mathrm{Ind}_F^E(M)) \simeq H^{\ast}(E,M)$ by Shapiro's lemma (\cite[Prop.~1.6.4]{NSW08}). 
    
    For each place $v\not\in S$, $X_v$ has good reduction and hence 
    $KH_1^{(1)}(X_v,\Z/p) = 0$ by \autoref{conj:Kato2}, and \autoref{conj:Kato} (Kato's theorem \cite[Sect.~2]{Kat86}, see also \cite[Cor.~2.9]{Kat86}).
    By \autoref{conj:Kato3}, the corestriction maps induce the commutative diagram
    \begin{equation}
        \label{diag:KH}
    \vcenter{
    \xymatrix{
    KH_1^{(1)}(X_E,\Z/p) \ar[d]\ar[r]^-{\simeq} & \displaystyle \bigoplus_{v\in S}\bigoplus_{w\mid v} KH_1^{(1)}(X_w,\Z/p)\ar[d] \\
    KH_1^{(1)}(X,\Z/p) \ar[r]^-{\simeq} & \displaystyle\bigoplus_{v\in S} KH_1^{(1)}(X_v,\Z/p),
    }}
    \end{equation}
    where $X_w := X_E\otimes_E E_w$ for each place $w$ in $E$.
    By \cite[Cor.~2.9]{Kat86}, for each place $v$, 
    $KH_1^{(1)}(X_v,\Z/p)$ is the cokernel of the reciprocity map 
    $CH^{2}(X_v,1)/p \hookrightarrow \pi_1^{\ab}(X_v)/p \simeq H^1(X_v,\Omega_{\log}^2) = H^3(X_v,\Z/p(2))$. 
    By $X_v(F_v)\neq \emptyset$, we also have the short exact sequence 
    \[
    0\to V(X_v)/p \to \pi_1^{\ab}(X_v)^{\geo}/p\to KH_1^{(1)}(X_v,\Z/p) \to 0.
    \]
    For each place $v\in S$ in $F$, 
    by the assumptions on the Jacobian variety $\Jac(X_v)$, 
    the corestriction $\Cor_{w\mid v} \colon KH_1^{(1)}(X_w,\Z/p) \to KH_1^{(1)}(X_v,\Z/p)$ is bijective 
    for all place $w\mid v$ (by the proof of \autoref{lem:GaVXlocal}). 
    For a place $v\in S$, 
    the Galois group $G := \Gal(E/F)$ acts the set of the places $w$ in $E$ above $v$, 
    we have 
    $\left(\bigoplus_{w\mid v}KH_1^{(1)}(X_w,\Z/p)\right)_{G} = KH_1^{(1)}(X_{w(v)},\Z/p)$ 
    for some place $w(v)$ in $E$ above $v$. 
    By the commutative diagram \eqref{diag:KH}, 
    $KH_1^{(1)}(X_E,\Z/p)_G \to KH_1^{(1)}(X,\Z/p)$ is bijective. 
    From \eqref{diag:Cor}, 
    \[
    \xymatrix{
    &\left(V(X_E)/p\right)_G \ar[d]^{N_{E/F}}\ar[r] & H^1(E,M)_G \ar[r]\ar@{->>}[d]^{\Cor_{E/F}} & KH_1^{(1)}(X_E,\Z/p)_G \ar[d]^{\simeq}\ar[r] & 0 \\
    0 \ar[r] & V(X)/p \ar[r] & H^1(F,M) \ar[r] & KH_1^{(1)}(X,\Z/p)\ar[r] & 0 .
    }
    \]
    The commutative diagram above implies 
    that the norm map $V(X_E)/p \to V(X)/p$ is surjective, and thus 
    we obtain $\Big(\Ga\otimesM \underline{V}(X)\Big)(F)/\wp = 0$.
\end{proof}

 \begin{thm}\label{thm:global}
 Let $X =  X_1\times \cdots \times X_d$ be the product of 
 projective smooth and geometrically irreducible curves $X_1,\ldots, X_d$ over a global field $F$. 
 We assume the condition $(\mathrm{Pt})$ in \autoref{lem:isom} for each $X_i$ (see \autoref{rem:index} below).
 
\begin{enumerate}
	\item 
    The conditions $(\mathrm{Pt})$ for $X$ and $(\mathrm{Van}_1)$ hold. 
    In particular, 
    for any $r\ge 1$, we have isomorphisms 
    \begin{gather}
            \Big(W_r\otimesM \underline{CH}^{d+1}(X,1)\Big)(F)/\wp \simeq KH_0^{(1)}(X,\Z/p^r),\quad \mbox{and}\\
            f_{\ast}^{KH}\colon KH_0^{(1)}(X,\Z/p^r) \xrightarrow{\simeq} H_{p^r}^{2}(F).
    \end{gather}
    \item
	There is a short exact sequence 
	\[
	0 \to KH_0^{(1)}(X,\Z/p^r) \to \bigoplus_v KH_0^{(1)}(X_v,\Z/p^r) \to \Z/p^r\to 0.
	\]
\end{enumerate}
\end{thm}
\begin{proof}
(1) 
    As we assume the condition (Pt) for each $X_i$, 
    the condition (Pt) for $X$ holds.  
    Since we have
	$\underline{A}^{d+1}(X,1) = \underline{V}(X)$,
    \autoref{lem:GaVX} and \autoref{lem:GaVXprod} 
    imply $(\Ga\otimesM \underline{A}^{d+1}(X,1))(F)/\wp  = 0$.
    
\noindent
\smallskip
(2)
By (1) and \autoref{thm:local}, we have isomorphisms 
\begin{align*}
  f_{\ast}^{KH}\colon  KH_0^{(1)}(X,\Z/p^r\Z) &\xrightarrow{\simeq} KH_0^{(1)}(\Spec(F),\Z/p^r\Z) \simeq \Br(F)[p^r],\quad \mbox{and}\\
f_{\ast}^{KH}\colon KH_0^{(1)}(X_v,\Z/p^r\Z) &\xrightarrow{\simeq }KH_0^{(1)}(\Spec(F_v),\Z/p^r\Z) \simeq \Br(F_v)[p^r]. 
\end{align*}
The short exact sequence (the Hasse-Brauer-Noether theorem) 
\[
0 \to \Br(F)[p^r] \to \bigoplus_v \Br(F_v)[p^r] \to \Z/p^r \to 0
\]
gives the required exact sequecne.
\end{proof}

\begin{rem}\label{rem:index}
\begin{enumerate}
    \item 
   The classical theorem due to F.~K.~Schmidt asserts that the existence of a degree 1 cycle on a curve $X_k$ 
   over a finite field $k$. 
   For a global field $F$ over $k$, and the constant field extension $X = X_k\otimes_k F$, 
   the condition (Pt) holds.

\item 
    For a projective smooth and geometrically irreducible curve $X$ over $F$, 
    the degree of the canonical divisor is $2g-2$, where $g$ is the arithmetic genus of $X$. 
    This implies that the greatest common divisor 
    $\mathrm{gcd}\set{[F(x):F] | x \in X_0}$ 
    divides $2g-2$. 
    For example, if $p\ge 3$ and $p\nmid (g-1)$, then 
    the condition (Pt) holds. 
    On the other hand, Artin-Schreier curves over $F$ do not satisfy the condition (Pt).
\end{enumerate}
\end{rem}


\section{Reciprocity sheaves}
\label{sec:RS}
In this section we assume that the base field $k$ is perfect.
By a series of works of Kahn-Miyazaki-Saito-Yamazaki \cite{KSY16, KMSY1, KMSY2, KSY22} to generalize the theory of $\bA^1$-invariant sheaves to the theory of cube invariant sheaves and reciprocity sheaves, the category $\Cork$ of finite correspondences over $k$ is generalized to the category of \emph{proper modulus pairs} over $k$.
Koizumi-Miyazaki \cite{KM24} furthermore extended it to \emph{proper $\Q$-modulus pairs} and gave a motivic construction of de Rham-Witt complex.

We first recall the works of Kahn-Miyazaki-Saito-Yamazaki/Koizumi-Miyazaki, and we consider expressing the theorems (\autoref{thm:finite}, \autoref{thm:local}, \autoref{thm:global}) in terms of reciprocity sheaves.

\subsection*{\texorpdfstring{$\Q$}{Q}-Modulus pairs and modulus presheaves with transfers}
\begin{dfn}[{\cite[Sect.~1.1]{KMSY2}, \cite[Def.~2.1]{KM24}}]
A \textbf{proper $\Z$-modulus} (resp.~\textbf{$\Q$-modulus}) \textbf{pair}
is a pair $\sX=(X, D_X)$ of a proper separated scheme $X$ of finite type over $k$ and 
an effective Cartier 
(resp.~a $\Q$-effective $\Q$-Cartier) divisor $D_X$ on $X$ such that $X^\circ:=X\smallsetminus |D_X|$ is smooth over $k$.
\end{dfn}

Fix $\Lambda\in \set{\Z, \Q}$.
We define $\MCor_k^\Lambda$ to be the category of proper $\Lambda$-modulus pairs.
The canonical functor 
$\MCor_k^{\Z} \to \MCor_k^{\Q}$ is fully faithful (\cite[Cor.~1.3]{KM24}).
There is a natural functor
\begin{align*}
    \omega:&\MCor_k^\Lambda\to \Cork,\quad \sX\mapsto X^\circ.
\end{align*}
The category 
$\MCor_k^\Lambda$ has a symmetric monoidal structure given by
\[
\mathscr{X}\otimes \mathscr{Y}=(X\times Y, \mathrm{pr}_1^*D_X +\mathrm{pr}_2^*D_Y),
\]
and the functor $\omega$ is symmetric monoidal (cf.~\cite[Prop.~2.1.6]{KMSY3} for $\Lambda=\Z$ case).

Let $\PSh(\MCor_k^\Lambda)$ (resp.~$\PSh(\Cork)$) denote the category of additive presheaves on $\MCor_k^\Lambda$ (resp.~$\Cork$).
A presheaf $\cF\in \PSh(\MCor_k^\Lambda)$ is called a \textbf{$\Lambda$-modulus presheaf with transfers}.
By the Yoneda embedding $\Z_\mathrm{tr}: \uMCor_k^\Lambda\to \PSh(\MCor_k^\Lambda)$ and colimits, the monoidal structure on $\MCor_k^\Lambda$ is extended to a symmetric monoidal structure on $\PSh(\MCor_k^\Lambda)$.
The functor $\omega$ induces an adjunction
\begin{align*}
    \omega_!:& \PSh(\uMCor_k^\Lambda)\rightleftarrows \PSh(\Cork):\omega^*.
\end{align*}

\subsection*{Cube-invariance and reciprocity}
The object $\bcube:=(\bP^1, [\infty])\in \MCor_k^\Lambda$ is called the ``cube'' over $k$, which plays a fundamental role. 

\begin{dfn}\label{dfn:h0cube}
A presheaf $\cF\in \PSh(\MCor_k^\Lambda)$ is said to be \textbf{$\bcube$-invariant} (or \textbf{cube invariant}) if 
\[
\mathrm{pr}_1^*: \cF(\mathscr{X})\to\cF(\mathscr{X}\otimes \bcube)
\]
is an isomorphism for all $\mathscr{X}\in \MCor_k^\Lambda$.
\end{dfn}

For $\epsilon=0, 1$, let $i_\epsilon:(\Spec(k),\emptyset)\to \bcube$ be the morphism of proper modulus pairs given by the morphism of schemes $\epsilon: \Spec(k)\to \bP^1.$
Similar to the 0-th Suslin homology for $\PSh(\Cork)$, we define the \emph{cube-localization} and \emph{cube-invariant part} as follows.
\begin{dfn}
 Let $\cF\in \PSh(\MCork^\Lambda)$.
We define the \textbf{cube-localization of $\cF$} as
\[
h_0^\bcube(\cF)(\mathscr{X})=\Coker(\cF(\mathscr{X}\otimes\bcube)\xrightarrow{i_0^*-i_1^*}\cF(\mathscr{X})),
\]
and the \textbf{cube-invariant part of $\cF$} as 
\[
h^{0,\bcube}(\cF)(\mathscr{X})=\Hom(h_0^\bcube(\mathscr{X}),\cF).
\]
\end{dfn}
\begin{rem}
For $\cF\in \PSh(\MCor_k^\Lambda)$, $h_0^\bcube(\cF)$ is the maximal $\bcube$-invariant quotient of $\cF$ and there is a canonical surjection $\cF\twoheadrightarrow h_0^\bcube(\cF)$.
$h^{0,\bcube}(\cF)$ is the maximal cube-invariant subobject of $\cF$.
The functor $h^{0,\bcube}$ is right adjoint to $h_0^\bcube.$
\end{rem}

We write $h_0, h^0$ for the composite of functors respectively
\begin{align}
\label{eq:h0}
h_0& :\PSh(\MCor_k^\Lambda)\xrightarrow{h_0^\bcube}\PSh(\MCor_k^\Lambda)\xrightarrow{\omega_!} \PSh(\Cork),\\
h^0& :\PSh(\Cork)\xrightarrow{\omega^*}\PSh(\MCor_k^\Lambda) \xrightarrow{h^{0,\bcube}}\PSh(\MCor_k^\Lambda).
\end{align}

For $\sX\in \MCor_k^\Lambda$, we write $h_0^{\bullet}(\sX):=h_0^{\bullet}(\Z_\mathrm{tr}(\sX))$ for $\bullet \in \set{\bcube, \emptyset}$. 

\begin{dfn}\label{dfn:modulus}
Let $\cF\in \PSh(\Cork)$.
Let $U$ be a smooth scheme of finite type over $k$ and $a\in \cF(U)$.
Take a proper $\Lambda$-modulus pair $\sX=(X,D_X)$ with $X^\circ\simeq U.$
We say that \textbf{$a$ has a $\Lambda$-modulus} $\sX$ (or $\sX$ is a \textbf{$\Lambda$-modulus for $a$}) if the morphism $a\colon \Z_\mathrm{tr}(U)\to \cF$ factors through $\Z_\mathrm{tr}(U)\to h_0(\sX)$. 
\end{dfn}

\begin{lem}[\text{\cite[Lem.~3.20]{KM24}}]
 Let notation be as in \autoref{dfn:modulus}.
 A section $a\in \cF(U)$ has a $\Z$-modulus if and only if it has a $\Q$-modulus.
\end{lem}

By this lemma, having a modulus does not depend on $\Lambda \in \{\Z,\Q\}$.

\begin{dfn}
Let $\cF\in \PSh(\Cork)$.
We say that $\cF$ has \textbf{reciprocity} (or $\cF$ is a \textbf{reciprocity presheaf}) if every section of $\cF$ has a $\Lambda$-modulus.
We define $\RSC_k$ to be the full subcategory of $\PSh(\Cork)$ consisting of reciprocity presheaves.
\end{dfn}

A relation between $\bcube$-invariance and reciprocity is known as follows.
\begin{lem}[\text{\cite[Lem.~3.21]{KM24}}]\label{lem:cube-rec}
   If $\cF\in \PSh(\MCor_k^\Lambda)$ is $\bcube$-invariant then $\omega_!\cF\in \RSC_k$
\end{lem}

Recall that $\cF\in \PSh(\Cork)$ is said to be $\bA^1$-invariant if
\[
\mathrm{pr}_1^*: \cF(X)\to\cF(X\times \bA^1)
\]
is an isomorphism for all smooth $X$.
Let $\HI_k$ be the full subcategory of $\PSh(\Cork)$ consisting of $\bA^1$-invariant presheaves.

\begin{thm}[\text{\cite{KSY16, KSY22}}]
The category $\RSC_k$ contains the following objects.
\begin{enumerate}[label=$(\mathrm{\alph*})$]
\item
$\bA^1$-invariant presheaves.
\item
Any smooth commutative algebraic group $($regarded as an object of $\PSh(\Cork))$.
\item
De {\red Rham}-Witt sheaf $W_r\Omega^n$. \linelabel{com:23}
\end{enumerate}
\end{thm}

Put $\Z:=\Ztr(\Spec(k),\emptyset)$.
The objects $\widehat{\bW}_n^+, \Gm^+$ in $\PSh(\MCor_k^\Q)$ are defined as follows
\[
\widehat{\bW}_n^+:=\varinjlim_{\epsilon>0}\Z_\mathrm{tr}(\bP^1, \epsilon[0]+(n+\epsilon)[\infty]), \quad 
\Gm^+:=\varinjlim_{\epsilon>0}\Coker(i_1:\Z\to\Z_\mathrm{tr}(\bP^1, \epsilon[0]+\epsilon[\infty]))
\]
and $W_r^+$ is defined to be the quotient of $\widehat{\bW}_{p^{r-1}}^+\otimes \Z_{(p)}$ by the image of the idempotents $\ell^{-1}V_\ell F_\ell$ for all prime $\ell\neq p$, where $V_\ell$ and $F_\ell$ are the Verschiebung
 and the Frobenius for $\widehat{\bW}_n^+$.
By identifying the group schemes  $\Gm$ and $W_r$ with the objects of $\PSh(\Cork)$ represented by $\Gm$ and $W_r$ respectively, we have the following theorem.

\begin{thm}[{\cite[Cor.~5.12, Rmk.~6.4, Thm.~6.21]{KM24}}]\label{thm:KM}
   \begin{enumerate}
 \item
   There is an isomorphism
   \[
     h_0(\Gm^+)\simeq \Gm
   \]
    in $\PSh(\Cork)$.
\item
   Suppose that $\mathrm{ch}(k)=p>0$. For any $r\geq 1$, there is an isomorphism
   \[
   h_0(W_r^+)\simeq W_r
   \]
    in $\PSh(\Corkaff)$, 
    where $\Corkaff$ is the full subcategory of $\Cork$ consisting of affine schemes. 
\item
   Suppose that $\mathrm{ch}(k)=p\geq 3$.
   There is an isomorphism
   \[
   a_\mathrm{Nis}h_0(W_r^+\otimes (\Gm^+)^{\otimes n})
   \simeq
   W_r\Omega^n
   \]
    in the category $\Sh_\Nis(\Cork)$ of Nisnevich sheaves on $\Cork$,
    where $a_\mathrm{Nis}$ is the Nisnevich sheafification.
\end{enumerate}
\end{thm}

\begin{rem}
For $r=1$, that is, for the $W_1=\Ga$ case, another motivic construction of the sheaf of K\"ahler differentials $\Omega^n$ is known as follows.
\begin{thm}[{\cite[Thm.~5.19]{RSY22}}]
\label{thm:RSY}
  Suppose that $\mathrm{ch}(k)\neq 2,3,5.$
  Let $\Ga^M, \Gm^M$ be the following objects in $\PSh(\MCor_k^\Z)$
  \[
  \Ga^M:=\Coker(i_0:\Z\to\Z_\mathrm{tr}(\bP^1, 2[\infty])), \quad 
  \Gm^M:=\Coker(i_1:\Z\to\Z_\mathrm{tr}(\bP^1, [0]+[\infty])),
  \]
  where $\Z:=\Z_\mathrm{tr}(\Spec(k),\emptyset).$
  Then there is an isomorphism in $\Sh_{\Nis}(\Cork)$
  \[
  a_\Nis h_0(\Ga^M\otimes (\Gm^M)^{\otimes n})\simeq \Omega^n.
  \]
\end{thm}

\end{rem}

\subsection*{Relation with Mackey products and reciprocity}
For any $\cF_1, \cF_2\in \PSh(\MCork^\Lambda)$ we have a canonical surjection (cf.~\cite[Sect.~3.11]{RSY22}, \cite[Rmk.~3.14]{KM24})
\begin{equation}
h_0(\cF_1)\otimes h_0(\cF_2)\twoheadrightarrow h_0(\cF_1\otimes \cF_2).
\end{equation}
and thus for any field extension $F/k$ we have a canonical surjection
\begin{equation}
\label{eq:epi-tensor}
\Big(h_0(\cF_1)\otimesM h_0(\cF_2)\Big)(F)\simeq (h_0(\cF_1)\otimes h_0(\cF_2))(F)\twoheadrightarrow h_0(\cF_1\otimes \cF_2)(F).
\end{equation}
Here, the first isomorphism comes from the fact that the value at $F$ the tensor product in $\PSh(\Cork)$ coincides with that of the Mackey tensor product (cf.~\cite[Rmk.~4.1.3]{IR17}). 
In case $\cF_i=h^0(G_i)\in \PSh(\MCor_k^\Z) \ (G_i\in \HI_k)$ for all $i$, then we have isomorphisms
\[
h_0(\cF_1\otimes\cdots\otimes \cF_n)(F)
\simeq 
(G_1\overset{\HI}{\otimes}\cdots \overset{\HI}{\otimes}G_n)(F)
\simeq
K(F;G_1, \dots, G_n)
\]
where $\overset{\HI}{\otimes}$ means the tensor product in $\HI_k$.
The first isomorphism is proved in \cite[Theorem 1.6]{RSY22} and the second isomorphism is proved in \cite{KY13}.
In this case, \eqref{eq:epi-tensor} gives
\begin{align}\label{eq:epi-hi-case}
\Big(G_1\otimesM G_2\Big)(F)\simeq (G_1\otimes G_2)(F)\twoheadrightarrow h_0(\cF_1\otimes \cF_2)(F)\simeq K(F; G_1, G_2).
\end{align}
Hence the kernel of the map \eqref{eq:epi-hi-case} consists of elements coming from Weil reciprocity law in the definition of Somekawa $K$-group.

In general, $h_0(\cF_1\otimes \cF_2)$ is not in $\HI_k$ but we always have $h_0(\cF_1\otimes \cF_2)\in \RSC_k$ by \autoref{lem:cube-rec}.
A result of \cite[Thm.~4.9]{RSY22} gives a description of $h_0(\cF_1\otimes \cF_2)(F)$ in terms of generators and relations, where the relations are coming from Weil reciprocity law with modulus.
For this reason, the kernel of \eqref{eq:epi-tensor} is thought of as the part of Weil reciprocity law with modulus (or Rosenlicht-Serre reciprocity law).

\begin{dfn}\label{def:van}
Let $F$ be a field extension over $k$. 
For $\cF_1, \dots, \cF_n\in \PSh(\MCor_k^\Lambda)$, put $G:=h_0(\cF_1)\otimesM\cdots \otimesM h_0(\cF_n)$.
Let $\varphi$ be an endomorphism of $G$ as Mackey functors.
From \eqref{eq:epi-tensor} we have  the following map
\[
\pi: G(F)\twoheadrightarrow h_0(\cF_1\otimes\cdots\otimes \cF_n)(F).
\]
We say that \textbf{the reciprocity for $G$ is killed by $\varphi$} over $F$ if the kernel $\Ker(\pi)$ is zero in $\Coker(\varphi)(F)$.
\end{dfn}

\begin{prop}
    For any $m\ge 1$, 
    the reciprocity for $\Gm^{\otimesM n}$ is killed by the multiplication by $m$.
\end{prop}
\begin{proof}
    We have the following commutative diagram 
    \[
    \xymatrix{
    (\Gm^{\otimesM n})(F)\ar@{->>}[d]_{\pi} \ar@{->>}[rr] &&   (\Gm^{\otimesM n})(F)/m \ar[d]^{\simeq}\\
    h_0((\Gm^M)^{\otimes n})(F) \ar[r]^-{\simeq}_-{(\star)} & K_n^M(F) \ar@{->>}[r]  & K_n^M(F)/m, 
    }
    \]
    where the map ($\star$) is an isomorphism by \cite[Prop.~5.6]{RSY22} and
    the right vertical map is bijective by Kahn's theorem (\cite[Thm.~4.5]{Hir24}). 
    From the above diagram, 
    $\Ker(\pi)$ becomes zero in $(\Gm^{\otimesM n})(F)/m$.
\end{proof}

We now consider the reciprocity for $W_r\otimesM \Gm^{\otimesM n}, W_r\otimesM \underline{CH}_0(X)$ and prove the following.
\begin{cor}\label{cor:killed-rec}
\begin{enumerate}
   \item 
   Let $F$ be a field extension of $k$ of $\mathrm{ch}(F)=p\geq 3$. 
   The reciprocity for $W_r\otimesM \Gm^{\otimesM n}$ is killed by $\wp\otimes \mathrm{id}^{\otimes n}$ over $F$.
   \item
   Suppose that $k$ is a finite field and $F$ is a finite extension of $k$.
   Let $X$ be a projective smooth and geometrically irreducible scheme over $k$.
   The reciprocity for $W_r\otimesM  \underline{CH}_0(X)$ is killed by $\wp\otimes \mathrm{id}$ over $F$.
   \item
    Suppose that $k$ is a finite field, 
    and $F$ is a local field with residue field $k$ or a global field over $k$.
   Let $X_k$ be a projective smooth and geometrically irreducible curve over $k$ 
   and put $X = X_k\otimes_k F$. 
   The reciprocity for $W_r\otimesM  \underline{CH}^2(X,1)$ is killed by $\wp\otimes \mathrm{id}$ over $F$.
\end{enumerate}
\end{cor}

\begin{proof}
(1) 
    From \autoref{thm:H24} and \autoref{thm:KM} (3), we obtain the following commutative diagram
    \[
    \xymatrix{
    \Big(W_r\otimesM \Gm^{\otimesM n}\Big)(F)\ar@{->>}[d]_{\pi}\ar@{->>}[rr] && \Big(W_r\otimesM \Gm^{\otimesM n}\Big)(F)/\wp\ar[d]^{\simeq }_{\quad \tilde{s}_{F,p^r}^n}\\
    h_0(W_r^+\otimes (\Gm^+)^{\otimes n})(F)\ar[r]^{\hspace{1.3cm} \simeq }&W_r\Omega_F^n\ar@{->>}[r]& H_{p^r}^{n+1}(F).
    }
    \]
    Now, an easy diagram chase shows $\Ker(\pi)$ is zero in  $\Big(W_r\otimesM \Gm^{\otimes n}\Big)(F)/\wp$.

\noindent 
(2) 
    The structure map $f\colon X\to \Spec(k)$ induces a canonical morphism $h_0^\bcube(X,\emptyset)\to \Z$ in $\PSh(\MCor_k^\Z)$, and we have the following commutative diagram
    \[
    \xymatrix{
    \Big(W_r\otimesM \underline{CH}_0(X)\Big)(F)
    \ar[r]^-{\Id\otimes f_*^{CH}}\ar@{->>}[d]^{\pi}&\Big(W_r\otimesM \Z\Big)(F)\ar@{=}[r]&W_r(F)\ar[d]_{\simeq }^{\text{Thm.\,\ref{thm:KM}}}\\
    h_0(W_r^+\otimes h_0^\bcube(X,\emptyset))(F)\ar[r]
    &h_0(W_r^+\otimes \Z)(F)\ar@{=}[r]&h_0(W_r^+)(F).
    }
    \]
    Since the map $\Id\otimes f_*^{CH}$ becomes an isomorphism after modulo $\wp\otimes \Id$ by \autoref{thm:finite}, $\Ker(\pi)$ is zero in $\Coker(\wp \otimes \Id)$.

\noindent
(3)    The canonical morphism $h_0^\bcube(X_k,\emptyset)\to \Z$ in $\PSh(\MCor_k^\Z)$ induces a morphism
    \[
    h_0^\bcube(\Gm^+\otimes h_0^\bcube(X_k,\emptyset))\to h_0^\bcube(\Gm^+)
    \]
    whose sections over $\Spec (F)$ is the map $f_*^{CH}\colon CH^2(X,1)\to F^\times$.
    We now have the following commutative diagram
    \[
    \xymatrix@C=3mm{
    \Big(W_r\otimesM \underline{CH}^2(X,1)\Big)(F)
    \ar[r]^-{\Id\otimes f_*^{CH}}\ar@{->>}[d]^{\pi}&\Big(W_r\otimesM \Gm\Big)(F)\ar@{->>}[r]&\Big(W_r\otimesM \Gm\Big)(F)/\wp\ar[d]_{\simeq }^{\text{Thm.\,\ref{thm:H24}}}\\
    h_0(W_r^+\otimes h_0^\bcube(\Gm^+\otimes h_0^\bcube(X_k,\emptyset)))(F)\ar[r]
    &h_0(W_r^+\otimes \Gm^+)(F)\ar@{->>}[r]&(W_r\Omega_F^1/\mathrm{d} V^{r-1}F)/\wp.
    }
    \]
    Since the map $\Id\otimes f_*^{CH}$
    becomes an isomorphism after modulo $\wp\otimes \Id$ by \autoref{thm:local} and \autoref{thm:global} (see also 
    \autoref{rem:index}), $\Ker(\pi)$ is zero in $\Coker(\wp \otimes \Id)$.
\end{proof}

\def\cprime{$'$}
\providecommand{\bysame}{\leavevmode\hbox to3em{\hrulefill}\thinspace}



\end{document}